\documentclass[10pt]{article}

\usepackage{amsmath,amssymb,amsthm,mathtools}
\usepackage{aliascnt}
\usepackage{bm}
\usepackage{geometry}
\usepackage{enumitem}
\usepackage{longtable}
\usepackage{graphicx}
\usepackage{tikz}
\usepackage{authblk}
\usepackage[colorlinks=false,hidelinks]{hyperref}
\usepackage[nameinlink,noabbrev]{cleveref}
\theoremstyle{definition}
\newtheorem{definition}{Definition}[section]
\theoremstyle{plain}
\newaliascnt{theorem}{definition}
\newtheorem{theorem}[theorem]{Theorem}
\aliascntresetthe{theorem}
\newaliascnt{proposition}{definition}
\newtheorem{proposition}[proposition]{Proposition}
\aliascntresetthe{proposition}
\newaliascnt{lemma}{definition}
\newtheorem{lemma}[lemma]{Lemma}
\aliascntresetthe{lemma}
\newaliascnt{corollary}{definition}
\newtheorem{corollary}[corollary]{Corollary}
\aliascntresetthe{corollary}
\theoremstyle{remark}
\newaliascnt{remark}{definition}
\newtheorem{remark}[remark]{Remark}
\aliascntresetthe{remark}
\newaliascnt{example}{definition}
\newtheorem{example}[example]{Example}
\aliascntresetthe{example}
\newaliascnt{question}{definition}
\newtheorem{question}[question]{Question}
\aliascntresetthe{question}

\crefname{section}{Section}{Sections}
\Crefname{section}{Section}{Sections}
\crefname{subsection}{Subsection}{Subsections}
\Crefname{subsection}{Subsection}{Subsections}
\crefname{appendix}{Appendix}{Appendices}
\Crefname{appendix}{Appendix}{Appendices}
\crefname{equation}{Equation}{Equations}
\Crefname{equation}{Equation}{Equations}
\crefformat{equation}{#2(#1)#3}
\Crefformat{equation}{#2(#1)#3}
\crefname{table}{Table}{Tables}
\Crefname{table}{Table}{Tables}
\crefname{definition}{Definition}{Definitions}
\Crefname{definition}{Definition}{Definitions}
\crefname{theorem}{Theorem}{Theorems}
\Crefname{theorem}{Theorem}{Theorems}
\crefname{proposition}{Proposition}{Propositions}
\Crefname{proposition}{Proposition}{Propositions}
\crefname{lemma}{Lemma}{Lemmas}
\Crefname{lemma}{Lemma}{Lemmas}
\crefname{corollary}{Corollary}{Corollaries}
\Crefname{corollary}{Corollary}{Corollaries}
\crefname{remark}{Remark}{Remarks}
\Crefname{remark}{Remark}{Remarks}
\crefname{example}{Example}{Examples}
\Crefname{example}{Example}{Examples}
\crefname{question}{Question}{Questions}
\Crefname{question}{Question}{Questions}

\newcommand{\tw}{\mathrm{tw}}

\newcommand{\F}{\mathbb{F}}
\newcommand{\Pone}{\mathbb{P}^{1}}
\newcommand{\GL}{\operatorname{GL}}
\newcommand{\PSL}{\operatorname{PSL}}

\newcommand{\SL}{\operatorname{SL}}
\newcommand{\Sq}{\operatorname{Sq}}

\newcommand{\Tr}{\operatorname{tr}}

\newcommand{\calD}{\mathcal{D}}
\newcommand{\calU}{\mathcal{U}}

\title{Twisted primitive group association schemes}
\author[1]{Akihiro Higashitani}
\author[1]{Masanari Kamiya}
\author[2]{Hirotake Kurihara}

\affil[1]{\footnotesize Department of Pure and Applied Mathematics, Graduate School of Information Science and Technology, Osaka University, Yamadaoka 1-5, Suita, Osaka 565-0871, Japan}
\affil[2]{\footnotesize Department of Applied Science, Yamaguchi University, 2-16-1 Tokiwadai, Ube 755-8611, Japan}
\date{}

\begin{document}
\maketitle

\begin{abstract}
We give results on the question of whether the intersection numbers of a primitive group association scheme determine it up to combinatorial isomorphism.
For $G=\PSL(2,q)$, where $q$ is an odd prime power with $q=11$ or $q\ge 17$, or $q=2^f$ with $f\ge3$, we construct a Schur partition that is algebraically isomorphic to the partition of $G$ into conjugacy classes but not combinatorially isomorphic to it.
Consequently, the corresponding primitive group association schemes are not determined up to combinatorial isomorphism by their intersection numbers; in particular, they are non-separable.

For $\mathfrak A_6$ and $\mathfrak A_8$, we also explicitly construct Schur partitions that are algebraically isomorphic to the corresponding partitions into conjugacy classes but not combinatorially isomorphic to them.
\end{abstract}



\section{Introduction}
A recurring theme in finite group theory is how much of a finite group can be recovered from its ordinary character table.
Character tables package conjugacy classes and irreducible representations into a finite numerical object, and they are standard data in the study of finite simple and almost simple groups \cite{IsaacsCharacterTheory,AtlasFiniteGroups}.
Thus one is led to ask whether the character table of a finite group determines the group up to isomorphism.
In general, the answer is no.
The dihedral group $D_8$ and the quaternion group $Q_8$ have the same character table although they are not isomorphic.
Nevertheless, this question remains meaningful for particular groups and families, and related questions also arise in the classification of finite simple groups.

Association schemes give a combinatorial version of this question.
They originated in the study of partially balanced incomplete block designs and the Bose--Mesner algebra, and later became a framework connecting algebraic combinatorics, coding theory, permutation groups, and representation theory \cite{BoseMesner1959,BannaiIto,BaileyAssociationSchemes}.
For a finite group $G$, the group association scheme $\mathfrak X(G)$ is obtained from the partition of $G$ into conjugacy classes.
Equivalently, its Bose--Mesner algebra is the center of the group algebra $\mathbb C G$, with distinguished basis given by the conjugacy classes.
The first eigenmatrix $P$ of $\mathfrak X(G)$ is the ordinary character table with the usual normalization,
\[
  P_i(\chi)=\frac{|C_i|\chi(g_i^{-1})}{\chi(1)}
  \qquad (g_i\in C_i),
\]
where the inverse appears because of the relation convention used below.
Thus it is natural to ask whether this algebraic multiplication table, or equivalently the intersection numbers of $\mathfrak X(G)$, determines $\mathfrak X(G)$ up to combinatorial isomorphism.
The same question for Schur partitions, equivalently S-rings over groups, gives a natural intermediate problem.
S-rings give the algebraic language for Cayley association schemes.
They are used in isomorphism problems for Cayley objects and in the study of permutation groups using coherent configurations \cite{MuzychukPonomarenko,EvdokimovPonomarenko2009}.

Several positive results are known.
For example, Tomiyama proved that the group association schemes of $\mathfrak A_5$ and $\PSL(2,7)$ are uniquely determined up to combinatorial isomorphism by their intersection numbers, and Tomiyama--Yamazaki obtained analogous results for symmetric groups \cite{TomiyamaA5,TomiyamaPSL27,TomiyamaYamazakiSn}.
Yoshiara and Terada also constructed negative examples using split and non-split extensions \cite{YoshiaraSameParameters,TeradaSameParameters}.
The groups in their examples are not simple.

For a group association scheme, closed subsets correspond to normal subgroups of $G$, so $\mathfrak X(G)$ is primitive precisely when $G$ is simple \cite{BannaiIto}.
Character tables have been calculated for several families of commutative association schemes associated with finite simple and classical groups \cite{BannaiSubschemes,BannaiKawanakaSongHecke,BannaiHaoSongOrthogonal,BannaiSongHaoWeiClassical,TanakaPGL}.
Thus finite simple groups provide natural primitive examples for this question.
Since the known positive results concern only a few small groups and special families, it is natural to ask whether primitivity forces an association scheme to be uniquely determined up to combinatorial isomorphism by its intersection numbers.
To the best of our knowledge, before the present work no group association scheme of a non-abelian finite simple group was known to fail this property, even when the algebraically isomorphic association scheme is required to be a Cayley association scheme.
The main purpose of this paper is to give such examples uniformly in an infinite family.
Our construction applies to $G=\PSL(2,q)$.
Let $q$ be a prime power and put $m=\frac{q-1}{(2,q-1)}$, where $(a,b)$ denotes the greatest common divisor of $a$ and $b$.
For each $k\in(\mathbb Z/m\mathbb Z)^\times$, we define a Schur partition $\calD_{q,k}$ of $G$ in \cref{sec:psl_star_switching}.

\begin{theorem}
\label{thm:intro-main}
The partition $\calD_{q,k}$ is algebraically isomorphic to the partition of $G$ into conjugacy classes.
Moreover, if $k\not\equiv\pm1\pmod m$, then $\calD_{q,k}$ is not combinatorially isomorphic to this partition.
\end{theorem}

This theorem is proved in \cref{sec:psl_star_switching}.
In particular, if $q$ is an odd prime power with $q=11$ or $q\ge17$, or if $q=2^f$ with $f\ge3$, the primitive group association scheme $\mathfrak X(\PSL(2,q))$ has an algebraically isomorphic Cayley association scheme which is not combinatorially isomorphic to it.
Moreover, the Cayley association schemes associated with $\calD_{q,k}$ are non-Schurian.
In other words, $\mathfrak X(\PSL(2,q))$ is not determined up to combinatorial isomorphism by its intersection numbers, even when the competing association scheme is restricted to a Cayley association scheme over the same group.
In particular, $\mathfrak X(\PSL(2,q))$ is non-separable.
The excluded prime powers are exactly those for $q=3,5,7,9,13$ in odd characteristic and $q=2,4$ in even characteristic.

We also construct, for $\mathfrak A_6$ and $\mathfrak A_8$, Schur partitions that are algebraically isomorphic but not combinatorially isomorphic.
Both constructions are obtained by applying \cref{lem:main}.
An exhaustive search over all groups of order at most $200$ finds $256$ nontrivial unordered partitions yielding new twists for $48$ group isomorphism types, with $\mathrm{SmallGroup}(108,15)$ as the smallest example.

The paper is organized as follows.
In \cref{sec:AS} we recall association schemes and group association schemes.
In \cref{sec:Sring} we recall Schur partitions, Cayley association schemes, algebraic isomorphisms, and local graphs.
In \cref{sec:psl_star_switching} we prove the construction theorem for $\PSL(2,q)$.
In \cref{sec:twisted_conjugacy} we give sufficient conditions for replacing two conjugacy classes by two new basic sets without changing the structure constants.
In \cref{sec:X(A_n)} we apply the result in \cref{sec:twisted_conjugacy} to alternating groups; it produces new partitions for $\mathfrak A_6$ and $\mathfrak A_8$, but not for $\mathfrak A_4$, $\mathfrak A_5$, $\mathfrak A_7$, or $\mathfrak A_9$.
In \cref{sec:other_finite_groups_twist} we apply the same construction to every group of order at most $200$.
Finally, in \cref{sec:summary} we summarize the results and list future questions.

\section{Definitions and basic facts on association schemes}
\label{sec:AS}
We recall notation and definitions; see \cite{BannaiIto}.
We write an association scheme using a partition $\{R_i\}_{i\in\mathcal{I}}$ of $X\times X$, that is, a family satisfying $X\times X=\bigcup_{i\in\mathcal{I}}R_i$ and $R_i\cap R_j=\emptyset$ for $i\neq j$.
We write such an association scheme as $\mathfrak{X}=(X,\{R_i\}_{i\in\mathcal{I}})$.

\begin{definition}
Let $X$ be a finite set, let $\mathcal{I}$ be a finite index set with $0\in\mathcal{I}$, and let $\{R_i\}_{i\in\mathcal{I}}$ be a partition of $X\times X$.
We call $\mathfrak{X}=(X,\{R_i\}_{i\in\mathcal{I}})$ an \emph{association scheme} on $X$ if the following conditions are satisfied:
\begin{enumerate}[label=$(AS\arabic*)$]
  \item \label{AS:R0} $R_0=\{(x,x)\mid x\in X\}$.
  \item \label{AS:transpose} For every $i\in\mathcal{I}$, there exists $i^{*}\in\mathcal{I}$ such that
  $R_{i^{*}}=\{(y,x)\mid (x,y)\in R_i\}$.
  \item \label{AS:pijk} For all $i,j,k\in\mathcal{I}$ and all $x,y\in X$ with $(x,y)\in R_k$,
  \[
    p_{ij}^k := \bigl|\{z\in X \mid (x,z)\in R_i,\ (z,y)\in R_j\}\bigr|
  \]
  is independent of the choice of $(x,y)\in R_k$.
\end{enumerate}
The integers $p_{ij}^k$ are called the \emph{intersection numbers}.
In addition,
\begin{enumerate}[label=$(AS\arabic*)$]
  \setcounter{enumi}{3}
  \item \label{AS:commutative} If $p_{ij}^k=p_{ji}^k$ for all $i,j,k\in\mathcal{I}$, then $\mathfrak{X}$ is called \emph{commutative}.
  \item \label{AS:symmetric} If $i^{*}=i$ for all $i\in\mathcal{I}$, then $\mathfrak{X}$ is called \emph{symmetric}.
\end{enumerate}
\end{definition}

We note that every symmetric association scheme is commutative.

\begin{definition}
Let $\mathfrak X=(X,\{R_i\}_{i\in\mathcal I})$ and
$\mathfrak X'=(X',\{R'_i\}_{i\in\mathcal I'})$ be association schemes.
\begin{enumerate}[label=$(\arabic*)$]
  \item A bijection $f:X\to X'$ is called a \emph{combinatorial isomorphism}
  from $\mathfrak X$ to $\mathfrak X'$ if for every $i\in\mathcal I$
  there exists $i'\in\mathcal I'$ such that
  \[
    (x_1,x_2)\in R_i \iff (f(x_1),f(x_2))\in R'_{i'}
    \qquad (x_1,x_2\in X).
  \]
  \item A bijection $\varphi:\mathcal I\to\mathcal I'$ is called an
  \emph{algebraic isomorphism} if, for all $i,j,k\in\mathcal I$,
  \[
    p_{ij}^{k}=p_{\varphi(i)\,\varphi(j)}^{'\varphi(k)},
  \]
  where $p_{ij}^{k}$ and $p_{ab}^{'c}$ are the intersection numbers of
  $\mathfrak X$ and $\mathfrak X'$, respectively.
\end{enumerate}
We say that $\mathfrak X$ and $\mathfrak X'$ are
\emph{combinatorially} (resp. \emph{algebraically}) \emph{isomorphic} if there exists a combinatorial (resp. algebraic) isomorphism
from $\mathfrak X$ to $\mathfrak X'$.
\end{definition}

\begin{remark}
Every combinatorial isomorphism induces an algebraic isomorphism on the relation classes.
An algebraic isomorphism need not be induced by a combinatorial isomorphism, even when its source and target are combinatorially isomorphic.
\end{remark}

An association scheme $\mathfrak X$ is called \emph{separable} if, for every
association scheme $\mathfrak X'$, every algebraic isomorphism from
$\mathfrak X$ to $\mathfrak X'$ is induced by a combinatorial isomorphism
from $\mathfrak X$ to $\mathfrak X'$.

\begin{remark}
Separability is stronger than being uniquely determined up to combinatorial isomorphism by the intersection numbers.
In fact, there are association schemes that are uniquely determined up to combinatorial isomorphism by their intersection numbers but are not separable.
One example is the rank-three association scheme obtained from a doubly regular tournament on $15$ vertices; see \cite{ZivAvEnumeration}.
\end{remark}

\begin{definition}
Let $\mathfrak X=(X,\{R_i\}_{i\in\mathcal I})$ be an association scheme.
A subset $\mathcal I'\subseteq\mathcal I$ with $0\in\mathcal I'$ is called
\emph{closed} if $\bigcup_{j\in\mathcal I'}R_j$ is an equivalence relation on $X$.
The association scheme $\mathfrak X$ is called \emph{primitive} if its only closed subsets are
$\{0\}$ and $\mathcal I$.
\end{definition}

\begin{definition}
For $i\in\mathcal{I}$, let $A_i$ be the adjacency matrix of $R_i$.
The complex vector space
$\mathcal{A}:=\mathrm{span}_{\mathbb{C}}\{A_i\mid i\in\mathcal{I}\}$
is closed under matrix multiplication, and is called the \emph{Bose--Mesner algebra} of $\mathfrak{X}$.
\end{definition}
By definition,
\[
  A_i A_j = \sum_{k\in\mathcal{I}} p_{ij}^k A_k
\]
holds.
This equality means that $\{A_i\}_{i\in\mathcal{I}}$ is a distinguished basis of the Bose--Mesner algebra.

From now on, we assume that $\mathfrak{X}$ is commutative.
Then the Bose--Mesner algebra $\mathcal{A}$ is a commutative semisimple algebra, and the adjacency matrices $A_i$ ($i\in\mathcal{I}$) can be simultaneously diagonalized.
Hence there exists an index set $\mathcal{J}$ and a basis $\{E_j\mid j\in\mathcal{J}\}$ of $\mathcal{A}$ consisting of primitive idempotents.
That is, they satisfy $E_jE_\ell=\delta_{j\ell}E_j$ and $\sum_{j\in\mathcal{J}}E_j=I_X$, and form a basis of pairwise orthogonal idempotents of $\mathcal{A}$.
Here $I_X$ denotes the identity matrix on $X$.
In particular, $\{A_i\}_{i\in\mathcal{I}}$ and $\{E_j\}_{j\in\mathcal{J}}$ are two bases of the same algebra $\mathcal{A}$, and $|\mathcal{I}|=|\mathcal{J}|$.
Since $\{A_i\}_{i\in\mathcal{I}}$ and $\{E_j\}_{j\in\mathcal{J}}$ are bases of $\mathcal{A}$, each $A_i$ can be written as a linear combination of the $E_j$ as follows.
\[
  A_i=\sum_{j\in\mathcal{J}} P_i(j) E_j\qquad (i\in\mathcal{I}).
\]
The matrix
\[
  P=(P_i(j))_{j\in\mathcal{J},\,i\in\mathcal{I}}
\]
is called the \emph{first eigenmatrix}.

\begin{definition}
Let $G$ be a finite group, and let $\{C_i\}_{i\in\mathcal{I}}$ be the set of its conjugacy classes.
Put
\[
R_i:=\{(x,y)\in G\times G\mid yx^{-1}\in C_i\}\qquad (i\in\mathcal{I}).
\]
Then $\{R_i\}_{i\in\mathcal{I}}$ is a partition of $G\times G$, and
$\mathfrak{X}(G):=(G,\{R_i\}_{i\in\mathcal{I}})$ is an association scheme on $G$.
We call it the \emph{group association scheme} $\mathfrak{X}(G)$.
\end{definition}

For a conjugacy class $C_i$, choose $g_i\in C_i$.
The primitive idempotents of the Bose--Mesner algebra of $\mathfrak X(G)$ are indexed by the irreducible characters $\chi\in\operatorname{Irr}(G)$, and, with the above convention for $R_i$, the eigenvalue of $A_i$ on the $\chi$-component is
\[
  P_i(\chi)=\frac{|C_i|\,\chi(g_i^{-1})}{\chi(1)}.
\]
Thus the first eigenmatrix of $\mathfrak X(G)$ is obtained from the ordinary character table of $G$ by multiplying the $C_i$-column by $|C_i|$, dividing the $\chi$-row by $\chi(1)$, and replacing each class by its inverse.
For this reason the first eigenmatrix of an association scheme is also often called its character table.

\begin{definition}
An association scheme $\mathfrak{X}=(X,\{R_i\}_{i\in\mathcal{I}})$ is called \emph{Schurian} if there exists a transitive permutation group $G\le \mathrm{Sym}(X)$ such that $\{R_i\}_{i\in\mathcal{I}}$ is the set of orbitals of $G$, that is, the set of orbits of $G$ on $X\times X$.
\end{definition}

\begin{example}
  \label{ex:group-schurian}
  The group association scheme $\mathfrak{X}(G)$ is always Schurian.
  Indeed, the orbitals of the action of $G\times G$ on $G$ defined by $(a,b)\cdot x:=axb^{-1}$ are precisely the relations $\{R_i\}_{i\in\mathcal{I}}$ of $\mathfrak{X}(G)$.
\end{example}

\begin{lemma}\label{lem:schurian-local-transitive}
Let $\mathfrak{X}=(X,\{R_i\}_{i\in\mathcal{I}})$ be a Schurian association scheme.
Fix $x\in X$, an index $i\in\mathcal{I}$, and a subset $\mathcal{E}\subseteq\mathcal{I}$.
Define a directed graph on
\[
  R_i(x):=\{y\in X\mid (x,y)\in R_i\}
\]
by declaring $(y,z)$ to be an arc whenever $(y,z)\in\bigcup_{j\in\mathcal{E}}R_j$.
Then this directed graph is vertex-transitive.
\end{lemma}

\begin{proof}
Let $G\le\mathrm{Sym}(X)$ be a transitive permutation group whose orbitals are the relations $R_i$.
For fixed $x\in X$, the orbits of the stabilizer $G_x$ on $X$ are exactly the sets $R_i(x)$.
Hence $G_x$ acts transitively on $R_i(x)$.
Moreover, since each $R_j$ is an orbital of $G$, the union $\bigcup_{j\in\mathcal{E}}R_j$ is invariant under $G$, and in particular under $G_x$.
Thus $G_x$ acts on the directed graph by automorphisms.
Therefore the graph is vertex-transitive.
\end{proof}

\section{S-rings}
\label{sec:Sring}
In this section we recall basic facts on Schur partitions, S-rings, and Cayley association schemes over finite groups.
Throughout, $G$ is a finite group and $e$ denotes the identity element of $G$.
For a subset $X\subseteq G$, we write
\[
\underline{X}:=\sum_{x\in X}x
\]
for the corresponding element of the group algebra $\mathbb{C}G$.
Our notation mainly follows \cite{MuzychukPonomarenko}.

\begin{definition}
A partition $\mathcal{S}$ of a finite group $G$ is called a \emph{Schur partition}, or an \emph{S-partition}, if it satisfies the following conditions:
\begin{enumerate}[label=(S\arabic*)]
  \item $\{e\}\in\mathcal{S}$;
  \item if $X\in\mathcal{S}$, then $X^{-1}:=\{x^{-1}\mid x\in X\}\in\mathcal{S}$;
  \item
  $\mathcal{A}(\mathcal{S}):=\operatorname{Span}_{\mathbb{C}}\{\underline{X}\mid X\in\mathcal{S}\}$
  is a subring of $\mathbb{C}G$.
\end{enumerate}
\end{definition}

\begin{definition}
A subring $\mathcal{A}$ of $\mathbb{C}G$ is called an \emph{S-ring} (Schur ring) over $G$ if there exists a Schur partition $\mathcal{S}$ such that
$\mathcal{A}=\mathcal{A}(\mathcal{S})$.
We regard such a Schur partition as the partition associated with $\mathcal{A}$, and call each of its parts a \emph{basic set} of $\mathcal{A}$.
The set of all basic sets is denoted by $\mathcal{S}(\mathcal{A})$.
Conversely, every Schur partition $\mathcal{S}$ gives an S-ring $\mathcal{A}(\mathcal{S})$, and therefore Schur partitions and S-rings are equivalent data.
\end{definition}

\begin{definition}
Let $\mathcal{A}$ be an S-ring over $G$, and let $\mathcal{S}(\mathcal{A})$ be its family of basic sets.
For $X,Y,Z\in\mathcal{S}(\mathcal{A})$, write
\[
\underline{X}\,\underline{Y}
=\sum_{Z\in\mathcal{S}(\mathcal{A})} c_{XY}^{Z}\,\underline{Z}.
\]
The integers $c_{XY}^{Z}$ are called the \emph{structure constants} of $\mathcal{A}$.
\end{definition}

\begin{definition}
An association scheme $\mathfrak{X}=(G,\{R_i\}_{i\in\mathcal{I}})$ on a finite group $G$ is called a \emph{Cayley association scheme} if it is invariant under the right regular action; that is, for every $g\in G$ and every $i\in\mathcal{I}$,
\[
(x,y)\in R_i \iff (xg,yg)\in R_i
\]
holds.
\end{definition}

Let $\mathcal{A}$ be an S-ring over $G$, and let $\mathcal{S}(\mathcal{A})$ be its family of basic sets.
For each $X\in\mathcal{S}(\mathcal{A})$, put
\[
R_X:=\{(x,y)\in G\times G\mid yx^{-1}\in X\}.
\]
Then
\[
\operatorname{Cay}(G,\mathcal{A})
:=(G,\{R_X\}_{X\in\mathcal{S}(\mathcal{A})})
\]
is a Cayley association scheme on $G$.
Conversely, for a Cayley association scheme $\mathfrak{X}=(G,\{R_i\}_{i\in\mathcal{I}})$, put
\[
X_i:=\{x\in G\mid (e,x)\in R_i\}\qquad (i\in\mathcal{I}).
\]
Then $\{X_i\mid i\in\mathcal{I}\}$ is a Schur partition of $G$.
The correspondence
\[
\mathcal{A}\longleftrightarrow \operatorname{Cay}(G,\mathcal{A})
\]
gives a one-to-one correspondence between S-rings over $G$ and Cayley association schemes on $G$.
Moreover, the structure constants of the S-ring determine the intersection numbers of the corresponding Cayley association scheme, up to the order convention
$p_{XY}^{Z}=c_{YX}^{Z}$.
We shall use this correspondence without further comment.
That is, a Schur partition $\mathcal S$ over $G$ will also be regarded as the Cayley association scheme $(G,\{R_X\}_{X\in\mathcal{S}})$.

\begin{example}
  Let $G$ be a finite group, and let $\mathcal{C}_G:=\{C_i\}_{i\in\mathcal{I}}$ be the family of its conjugacy classes.
  Since $\mathcal{C}_G$ gives a Schur partition of $G$, $\mathcal{A}(\mathcal{C}_G)$ is an S-ring over $G$.
  Moreover, the Cayley association scheme $\operatorname{Cay}(G,\mathcal{A}(\mathcal{C}_G))$ coincides with the group association scheme of $G$.
\end{example}

\begin{definition}
Let $\mathcal{A}$ be an S-ring over a group $G$, and let $\mathcal{A}'$ be an S-ring over a group $G'$.
\begin{enumerate}[label=(\arabic*)]
  \item A bijection $f:G\to G'$ is called a \emph{combinatorial isomorphism} from $\mathcal{A}$ to $\mathcal{A}'$ if, for every $X\in\mathcal{S}(\mathcal{A})$, there exists $X^{f}\in\mathcal{S}(\mathcal{A}')$ such that
  \[
    yx^{-1}\in X \iff f(y)f(x)^{-1}\in X^{f}
  \]
  holds for all $x,y\in G$.
  This is equivalent to a combinatorial isomorphism between $\operatorname{Cay}(G,\mathcal{A})$ and $\operatorname{Cay}(G',\mathcal{A}')$.
  \item A bijection between basic sets
  $\varphi:\mathcal{S}(\mathcal{A})\to \mathcal{S}(\mathcal{A}')$
  is called an \emph{algebraic isomorphism} if, for all $X,Y,Z\in\mathcal{S}(\mathcal{A})$,
  $c_{XY}^{Z}=c_{\varphi(X)\varphi(Y)}^{\varphi(Z)}$
  holds.
\end{enumerate}
We say that $\mathcal A$ and $\mathcal A'$ are
\emph{combinatorially} (resp. \emph{algebraically}) \emph{isomorphic} if there exists a combinatorial (resp. algebraic) isomorphism
from $\mathcal A$ to $\mathcal A'$.
We also use the terms combinatorially isomorphic and algebraically isomorphic for Schur partitions.
\end{definition}

Under the correspondence between S-rings and Cayley association schemes, these definitions coincide with the corresponding definitions of algebraic and combinatorial isomorphism for association schemes.

We next define local directed graphs and use them to obtain invariants under combinatorial isomorphism.

\begin{definition}
Let $G$ be a finite group, and let $\mathcal{S}$ be a Schur partition over $G$.
For a basic set $X\in\mathcal{S}$, a union $Y$ of basic sets, and an element $g\in G$, write $R_Y$ for the union of the relations $R_Z$ over the basic sets $Z\subseteq Y$.
Define $\Gamma_Y(X,g)$ to be the directed graph with vertex set $R_X(g)=Xg$
and with an arc $(a,b)$ whenever $(a,b)\in R_Y$, equivalently $ba^{-1}\in Y$.
We call it a \emph{local directed graph} with respect to $\mathcal{S}$.
We write $\Gamma_Y(X)$ for $\Gamma_Y(X,e)$.
In particular, $\Gamma_Y(X)$ is the directed graph with vertex set $X$ in which $(a,b)$ is an arc precisely when $ba^{-1}\in Y$.
\end{definition}

\begin{lemma}\label{lem:local-basepoint}
Let $G$ be a finite group, let $\mathcal{S}$ be a Schur partition over $G$, let $X\in\mathcal S$ be a basic set, and let $Y$ be a union of basic sets.
Then for all $g_1,g_2\in G$, the graphs $\Gamma_{Y}(X,g_1)$ and $\Gamma_{Y}(X,g_2)$ are isomorphic as directed graphs.
\end{lemma}

\begin{proof}
Consider the map $\phi:Xg_1\longrightarrow Xg_2$ defined by $\phi(a):=ag_1^{-1}g_2$.
Indeed, for every $xg_1\in Xg_1$ we have $\phi(x g_1)=xg_2\in Xg_2$, so this is a well-defined map.
It is immediate from the definition that $\phi$ is a bijection.
Moreover, for $x_1 g_1, x_2 g_1\in Xg_1$,
we have $(x_2 g_1)(x_1 g_1)^{-1}=\phi(x_2 g_1)\phi(x_1 g_1)^{-1}=x_2x_1^{-1}$.
Therefore $(x_1 g_1,x_2 g_1)$ is an arc of $\Gamma_Y(X,g_1)$ if and only if
$(\phi(x_1 g_1),\phi(x_2 g_1))$ is an arc of $\Gamma_Y(X,g_2)$.
Thus $\phi$ gives an isomorphism of directed graphs from $\Gamma_{Y}(X,g_1)$ to $\Gamma_{Y}(X,g_2)$.
\end{proof}

\begin{lemma}\label{lem:inv}
Let $\mathcal{A}$ be an S-ring over $G$, let $\mathcal{A}'$ be an S-ring over $G'$, and suppose that $f:G\to G'$ is a combinatorial isomorphism from $\mathcal{A}$ to $\mathcal{A}'$.
Let $X\in\mathcal{S}(\mathcal{A})$, and let $Y$ be a union of basic sets of $\mathcal{A}$.
Write $X^f$ for the corresponding basic set of $\mathcal{A}'$, and write $Y^f$ for the union of the corresponding basic sets of $\mathcal{A}'$.
Then for every $g\in G$,
\[
  \Gamma_{Y}(X,g)\cong \Gamma_{Y^f}(X^f,f(g))
\]
as directed graphs.
Consequently, the numbers of weakly and strongly connected components of $\Gamma_{Y}(X)$, as well as its number of vertices, are invariants of the Cayley association scheme obtained from the Schur partition under combinatorial isomorphism.
\end{lemma}

\begin{proof}
By the definition of combinatorial isomorphism, for all $a,b\in G$ we have
\[
  bg^{-1}\in X \iff f(b)f(g)^{-1}\in X^f.
\]
Thus $f$ maps the vertex set $Xg$ to $X^f f(g)$.
Moreover,
\[
  ba^{-1}\in Y \iff f(b)f(a)^{-1}\in Y^f,
\]
and so $f$ also preserves the arc relation.
Therefore $f|_{Xg}$ gives an isomorphism of local directed graphs.
\end{proof}

\section{The construction of the twisted group association scheme for \texorpdfstring{$\PSL(2,q)$}{PSL(2,q)}}
\label{sec:psl_star_switching}

In this section we prove the main construction of the paper.
Throughout this section, let $q$ be a prime power and put
\[
  G=\PSL(2,q)=\SL(2,q)/\{\pm I\}=\{[M]\mid M\in\SL(2,q)\}.
\]
Here $[M]=\{M,-M\}$, and we call the elements of $\{M,-M\}$ lifts of $[M]$.
In what follows, matrix representatives in $\SL(2,q)$ are displayed with parentheses, while elements of $\PSL(2,q)$ are displayed with brackets.
We also write
\[
  \Pone(\F_q)=\{[x:y]\mid (x,y)\in\F_q^2\setminus\{(0,0)\}\}.
\]
Here $[x:y]$ denotes the one-dimensional subspace spanned by the column vector $\begin{pmatrix}x\\y\end{pmatrix}$.
Let
\[
  \Sq(\F_q^\times)=\{a^2\mid a\in\F_q^\times\}
\]
be the image of the squaring map on $\F_q^\times$.
Then
\[
|\Sq(\F_q^\times)|=\frac{q-1}{(2,q-1)}
\]
is well known.
Thus, if $q$ is even, then the squaring map is an automorphism of $\F_q^\times$, so $\Sq(\F_q^\times)=\F_q^\times$; if $q$ is odd, then $\Sq(\F_q^\times)$ is the subgroup of squares.
Put $m=|\Sq(\F_q^\times)|$.

The group $G$ has the natural transitive left action on $\Pone(\F_q)$ given by
\[
\begin{bmatrix}
    a & b \\
    c & d
\end{bmatrix}
\cdot[x:y]
=
[ax+by : cx+dy].
\]

We recall the standard classification of conjugacy classes in $\PSL(2,q)$ which will be used below
\cite{FH1991,HuppertBlackburn,DixonMortimer}.
Let $g\in G$, choose a lift $\hat g\in\SL(2,q)$, and put
\[
  f_{\hat g}(X)=X^2-(\Tr \hat g)X+1.
\]
When $q$ is odd, the lift is determined only up to sign, but $(\Tr \hat g)^2$ is determined by $g$.
The conjugacy classes are of the following four types.
\begin{enumerate}[label=(\roman*)]
  \item The identity class.  This corresponds to $g=[I]$.

  \item The nonidentity unipotent classes.  These are the cases in which $\hat g\ne\pm I$ and $(\Tr \hat g)^2=4$.
  If $q$ is even, the nonidentity unipotent elements form one conjugacy class, represented by an element $u$ whose lift is
  \[
    \hat u=\begin{pmatrix}1&1\\0&1\end{pmatrix}.
  \]
  If $q$ is odd, they split into two conjugacy classes.
  Fix a nonsquare $d_0\in\F_q^\times$.  Representatives $u_1,u_{d_0}$ may be chosen with lifts
  \[
    \hat u_1=\begin{pmatrix}1&1\\0&1\end{pmatrix},
    \qquad
    \hat u_{d_0}=\begin{pmatrix}1&d_0\\0&1\end{pmatrix}.
  \]
  For a general unipotent element, after replacing $\hat g$ by $-\hat g$ if necessary, assume that $\Tr \hat g=2$.
  Choose a nonzero vector $v\in\F_q^2$ with $(\hat g-I)v=0$, and then choose $w\in\F_q^2$ with $\det(v,w)=1$.
  Then
  \[
    (\hat g-I)w=av
  \]
  for a uniquely determined $a\in\F_q^\times$, and the square class of $a$ determines the conjugacy class.
  This square class is independent of the auxiliary choices: replacing $v$ by $\lambda v$ and $w$ by $\lambda^{-1}w$, with $\lambda\in\F_q^\times$, replaces $a$ by $\lambda^{-2}a$.

  \item The split semisimple classes.  These are the cases in which $f_{\hat g}$ has two distinct roots over $\F_q$.
  Let the roots of $f_{\hat g}$ be $r,r^{-1}\in\F_q^\times$.
  A representative $t_r$ may be chosen with lift
  \[
    \hat t_r=\begin{pmatrix}r&0\\0&r^{-1}\end{pmatrix}.
  \]
  Put $s=r^2\in \Sq(\F_q^\times)$, and write $\bar s=\{s,s^{-1}\}$.
  We denote the conjugacy class containing such elements by $C_{\bar s}$.
  When $q$ is odd, note that
  \[
    (\Tr \hat g)^2=(r+r^{-1})^2=s+s^{-1}+2.
  \]

  \item The non-split semisimple classes.  These are the cases in which $f_{\hat g}$ is irreducible over $\F_q$.
  A representative $n_{a_0}$ may be chosen with lift
  \[
    \hat n_{a_0}=\begin{pmatrix}0&-1\\1&a_0\end{pmatrix},
  \]
  where $f_{a_0}(X)=X^2-a_0X+1$ is irreducible over $\F_q$.
\end{enumerate}
Thus the conjugacy class of a given element can be read off from $\Tr \hat g$ as follows.
When $q$ is odd, we use $(\Tr \hat g)^2$ to remove the sign ambiguity in the lift.
If $(\Tr \hat g)^2=4$, then the element lies in the identity class if $\hat g$ is scalar, and otherwise it is unipotent.
In the latter case, the two classes are distinguished by the square class of $a$ above.
If $(\Tr \hat g)^2\ne4$, then the element is semisimple; it is split semisimple if $(\Tr \hat g)^2-4$ is a square in $\F_q$, and non-split semisimple if $(\Tr \hat g)^2-4$ is a nonsquare.
In the split case, if the roots are $r,r^{-1}\in\F_q^\times$ and $s=r^2$, then the class is $C_{\bar s}$.

When $q$ is even, $\Tr \hat g$ itself is determined by the element of $G$.
If $\Tr \hat g=0$, then the element is the identity if $\hat g=I$, and otherwise lies in the unique nonidentity unipotent class.
If $\Tr \hat g\ne0$, then the element is semisimple; it is split semisimple if
\[
  f_{\hat g}(X)=X^2+(\Tr \hat g)X+1
\]
splits over $\F_q$, and non-split semisimple otherwise.

Fix $\infty:=[1:0]\in\Pone(\F_q)$,
and write its stabilizer as
\[
  B=G_\infty=\left\{\begin{bmatrix}
    \ell & c \\
    0 & \ell^{-1}
  \end{bmatrix}\mid \ell\in\F_q^\times,\ c\in\F_q\right\}.
\]
In the affine coordinate on $\Pone(\F_q)\setminus\{\infty\}$, the action of $B$ is
\[
  \begin{bmatrix}
    \ell & c \\
    0 & \ell^{-1}
  \end{bmatrix}
  \cdot[z:1]
=[\ell z+c:\ell^{-1}]
=[\ell^2z+c\ell:1].
\]
Thus $B$ determines affine maps of the form $z\mapsto \ell^2z+c\ell$ on $\F_q$.
Conversely, for any $s\in \Sq(\F_q^\times)$ and $a\in\F_q$, one can choose an element of $B$ inducing the affine map $z\mapsto sz+a$.
Indeed, choose $\ell\in\F_q^\times$ with $s=\ell^2$, and define
\[
  h_{s,a}:=
  \begin{bmatrix}
    \ell & a\ell^{-1} \\
    0 & \ell^{-1}
\end{bmatrix}.
\]
This is independent of the choice of $\ell$ and induces $z\mapsto sz+a$.
In this way we have the unique expression
\[
  B=\{h_{s,a}\mid s\in \Sq(\F_q^\times),\ a\in\F_q\}.
\]
The multiplication is given by
\begin{equation}
\label{eq:psl-borel-mult}
  h_{s,a}h_{t,b}=h_{st,a+sb}.
\end{equation}

For $s\in \Sq(\F_q^\times)$, define
\[
  H_s=\{h_{s,a}\mid a\in\F_q\}\subseteq B.
\]
Then $H_1$ is a Sylow $p$-subgroup of $G$ of order $q$, and, for $s\ne1$, $H_s$ consists of split semisimple elements in the Borel subgroup.
For $s\ne1$, we may reparametrize $a=(1-s)\beta$ gives
\[
  h_{s,(1-s)\beta}:z\longmapsto sz+(1-s)\beta,
  \qquad \beta\in\F_q,
\]
where $\beta$ is the other fixed point.
As $\beta$ runs over $\F_q$, these are exactly the elements of $H_s$.
Then
\begin{equation}
\label{eq:psl-class-borel-intersection}
  C_{\bar s}\cap B=H_s\cup H_{s^{-1}}.
\end{equation}
To see this, let $g\in C_{\bar s}\cap B$, and take a lift $\hat g\in\SL(2,q)$.
Since $g\in B$, the lift $\hat g$ may be chosen upper triangular.
If its diagonal entries are $\ell,\ell^{-1}$, then regardless of which diagonal entry occurs in the $(1,1)$-position, we have $\{\ell^2,\ell^{-2}\}=\bar s$.
Thus $g\in H_s\cup H_{s^{-1}}$.
Conversely, if $g\in H_s\cup H_{s^{-1}}$, then $g\in B$ by definition.
Moreover, if the diagonal entries of a lift of $g$ are $\ell,\ell^{-1}$, then $g\in H_s\cup H_{s^{-1}}$ means precisely that $\{\ell^2,\ell^{-2}\}=\bar s$.
Hence $g\in C_{\bar s}$, proving \cref{eq:psl-class-borel-intersection}.
Thus \cref{eq:psl-class-borel-intersection} precisely describes the intersection of each split semisimple conjugacy class with the Borel subgroup $B$.
For the other types of conjugacy classes, the identity class meets $B$ only in the identity, the unipotent classes meet $B$ in subsets of $H_1$, and the non-split semisimple classes do not meet $B$.

Choose $k\in(\mathbb Z/m\mathbb Z)^\times$.
Since $\Sq(\mathbb F_q^\times)$ is cyclic of order $m$ and $(k,m)=1$, the map $s\mapsto s^k$ is a permutation of $\Sq(\F_q^\times)$.
Using this permutation, define a new partition $\calD_{q,k}$ of $G$ as follows.
Leave the identity class, the unipotent classes, and the non-split semisimple classes unchanged.
For each split label $\bar s=\{s,s^{-1}\}$, replace $C_{\bar s}$ by
\begin{equation}
\label{eq:psl-switched-part}
  D_{\bar s}
  =\bigl(C_{\bar s}\setminus (H_s\cup H_{s^{-1}})\bigr)
    \cup H_{s^k}\cup H_{s^{-k}}.
\end{equation}
Since the power map permutes the unordered pairs $\{s,s^{-1}\}$, these sets together with the unchanged conjugacy classes form a partition of $G$.

\begin{theorem}
\label{thm:psl_star_switching}
Let $q$ be a prime power and let $G=\PSL(2,q)$.
For every $k\in(\mathbb Z/m\mathbb Z)^\times$,
the partition $\calD_{q,k}$ is a Schur partition of $G$, and $\calD_{q,k}$ is algebraically isomorphic to $\mathcal{C}_G$, the partition of $G$ into conjugacy classes.
Moreover, if $k\not\equiv\pm1\pmod m$,
then $\calD_{q,k}$ is not combinatorially isomorphic to $\mathcal{C}_G$.
\end{theorem}

\begin{corollary}
\label{cor:psl_prime_power_family}
For every odd prime power $q$ with $q=11$ or $q\ge17$, and for every even prime power $q=2^f$ with $f\ge3$, the primitive group association scheme $\mathfrak X(\PSL(2,q))$ is not determined up to combinatorial isomorphism by its intersection numbers.
More precisely, it admits an algebraically isomorphic Cayley association scheme over the same group that is not combinatorially isomorphic to it.
In particular, it is non-separable.
\end{corollary}

\begin{proof}
There is no $k\in(\mathbb Z/m\mathbb Z)^\times$ with $k\not\equiv\pm1\pmod m$ if and only if $m=1,2,3,4,6$.
Indeed, the nonexistence of such a $k$ is equivalent to every element of $(\mathbb Z/m\mathbb Z)^\times$ being equal to $\pm1$.
This is equivalent to $|(\mathbb Z/m\mathbb Z)^\times|\le2$, or $\varphi(m)\le2$, where $\varphi$ is Euler's phi function.
The positive integers with $\varphi(m)\le2$ are exactly $m=1,2,3,4,6$.
In odd characteristic these give $q=3,5,7,9,13$, and in even characteristic they give $q=2,4$.
Thus, for every odd prime power $q$ with $q=11$ or $q\ge17$, and every even prime power $q=2^f$ with $f\ge3$, there exists $k\in(\mathbb Z/m\mathbb Z)^\times$ such that $k\not\equiv\pm1\pmod m$.
For these parameters $\PSL(2,q)$ is simple, so the corresponding group association scheme is primitive.
The assertion follows from \cref{thm:psl_star_switching}.
\end{proof}

We call $\calD_{q,k}$ a \emph{twisted partition} of $\mathcal{C}_G$.
We call the Cayley association scheme obtained from $\calD_{q,k}$ a \emph{twisted group association scheme} of $\mathfrak X(\PSL(2,q))$.
We prove \cref{thm:psl_star_switching} in the rest of this section.

\subsection{Algebraic isomorphism}
In this subsection, we establish the counting lemmas needed for the first assertion of \cref{thm:psl_star_switching} and conclude by proving that $\calD_{q,k}$ is a Schur partition algebraically isomorphic to $\mathcal{C}_G$.

\begin{lemma}
\label{lem:psl_borel_counts}
Fix $s,t,u\in \Sq(\F_q^\times)$.
For every subset $A\subset H_s$ and every $z\in H_u$, we have
\[
  |\{(x,y)\in A\times H_t \mid xy=z\}|
  =
  \begin{cases}
    |A| & \text{if $u=st$;} \\
    0 & \text{if $u\neq st$.}
  \end{cases}
\]
\end{lemma}

\begin{proof}
By \eqref{eq:psl-borel-mult}, $H_sH_t\subseteq H_{st}$.
Thus if $u\ne st$, there is no $(x,y)\in A\times H_t$ with $xy=z\in H_u$.

Now suppose that $u=st$.
Fix $x\in A$, and write $x=h_{s,a}$ and $z=h_{st,c}$.
By \eqref{eq:psl-borel-mult}, the element $y=h_{t,b}\in H_t$ satisfying $xy=z$ is determined uniquely by
\[
  a+sb=c.
\]
Thus for each $x\in A$ there is a unique $y\in H_t$, and the desired number is $|A|$.
\end{proof}

\begin{lemma}
\label{lem:psl_mixed_counts}
Let $g\in G\setminus B$, let $s\in \Sq(\F_q^\times)\setminus\{1\}$, and let $C$ be a conjugacy class of $G$.
Then
\[
  |\{h\in H_s\mid gh\in C\}|
  =
  |\{h\in H_{s^k}\mid gh\in C\}|.
\]
The same assertion holds with $hg$ in place of $gh$.
\end{lemma}

\begin{proof}
Choose a lift
\[
  \hat g=\begin{pmatrix}a&b\\ c&d\end{pmatrix}\in\SL(2,q)
\]
of $g$.
Since $g\notin B$, we have $c\ne0$.
Choose $\ell\in\F_q^\times$ with $\ell^2=s$.
For $\beta\in\F_q$, choose the following lift of $h_{s,(1-s)\beta}$:
\[
  \widehat{h}_{s,(1-s)\beta}
  =
  \begin{pmatrix}
  \ell&(1-s)\beta \ell^{-1}\\
  0&\ell^{-1}
  \end{pmatrix}.
\]
Then
\[
\hat g\widehat{h}_{s,(1-s)\beta}
=
\begin{pmatrix}
a\ell & a(1-s)\beta \ell^{-1}+b\ell^{-1} \\
c\ell & c(1-s)\beta \ell^{-1}+d\ell^{-1}
\end{pmatrix},
\]
and hence
\begin{equation}
\label{eq:psl_mixed_trace}
  \Tr(\hat g\widehat{h}_{s,(1-s)\beta})
  =\frac{as+d+c(1-s)\beta}{\ell}.
\end{equation}
The numerator is a nonconstant affine function of $\beta$.

If $q$ is even, conjugacy in $\SL(2,q)=\PSL(2,q)$, away from the identity, is controlled by the trace.
As $\beta$ runs over $\F_q$, the right-hand side of \eqref{eq:psl_mixed_trace} runs once through all of $\F_q$, independently of $s$.
Since $g\notin B$, the product $\hat g\widehat{h}_{s,(1-s)\beta}$ is never the identity.
Thus the number of products lying in any conjugacy class is independent of $s$.

Now suppose that $q$ is odd.
Except for the two unipotent classes, projective conjugacy is controlled by $\Tr^2$.
Squaring \eqref{eq:psl_mixed_trace} gives
\[
  \Tr(\hat g\widehat{h}_{s,(1-s)\beta})^2
  =\frac{(as+d+c(1-s)\beta)^2}{s}.
\]
Since $s$ is a square and the numerator is the square of a nonconstant affine function, the distribution of the values of $\Tr(\hat g\widehat{h}_{s,(1-s)\beta})^2$ is independent of $s$: zero occurs once, each nonzero square occurs twice, and nonsquares do not occur.

It remains, in odd characteristic, to separate the two unipotent classes.
The condition $\Tr(\hat g\widehat{h}_{s,(1-s)\beta})=2\varepsilon$, with $\varepsilon=\pm1$, determines a unique $\beta$ for each sign.
For this product, the trace $2$ lift is $\varepsilon\hat g\widehat{h}_{s,(1-s)\beta}$.
Its lower-left entry is $\varepsilon c\ell$, and the square class invariant for a trace $2$ unipotent element is represented by the negative of the lower-left entry whenever that entry is nonzero.
Thus the corresponding unipotent class is determined by the square class of $-\varepsilon c\ell$.
If $-1$ is a nonsquare, then changing $\varepsilon$ interchanges the two square classes; hence the two signs contribute once to each unipotent class, independently of $s$.
If $-1$ is a square, then $q\equiv1\pmod 4$, and $m=(q-1)/2$ is even.
Hence every unit $k\in(\mathbb Z/m\mathbb Z)^\times$ is odd.
Replacing $\ell$ by $\ell^k$ therefore preserves its square class.
Thus replacing $H_s$ by $H_{s^k}$ preserves the counts for the two unipotent classes as well.

Finally, the elements $hg$ and $gh$ are conjugate in $G$, since $hg=h(gh)h^{-1}$.
The assertion for $hg$ follows.
\end{proof}

\begin{lemma}
\label{lem:psl_outside_outside}
For conjugacy classes $C_i,C_j$ of $G$, the function from $B$ to the nonnegative integers given by
\begin{equation}\label{eq:psl_outside_outside}
  z\longmapsto
  |\{(x,y)\in(C_i\cap(G\setminus B))\times(C_j\cap(G\setminus B))\mid xy=z\}|
\end{equation}
is constant on $H_s$, for $s\in \Sq(\F_q^\times)\setminus\{1\}$, and the constant on $H_s$ is equal to the constant on $H_{s^k}$.
\end{lemma}

\begin{proof}
The group $B$ acts transitively on $H_s$ by conjugation.
Indeed, conjugation by a translation $h_{1,c}\in H_1$ gives
\[
  h_{1,c}h_{s,a}h_{1,c}^{-1}
  =h_{s,a+(1-s)c}.
\]
Since $1-s\ne0$, varying $c$ moves the second coordinate to any element of $\F_q$.
Thus any $z,z'\in H_s$ are $B$-conjugate.
On the other hand, $C_i,C_j$ are $G$-conjugacy classes, and $G\setminus B$ is preserved by $B$-conjugation.
Hence, if $z'=bzb^{-1}$, then
\[
  (x,y)\longmapsto (bxb^{-1},byb^{-1})
\]
gives a bijection from
\[
  \{(x,y)\in(C_i\cap(G\setminus B))\times(C_j\cap(G\setminus B))\mid xy=z\}
\]
to
\[
  \{(x,y)\in(C_i\cap(G\setminus B))\times(C_j\cap(G\setminus B))\mid xy=z'\}.
\]
Therefore the function in \eqref{eq:psl_outside_outside} is constant on $H_s$.

We prove the second assertion.
Let $N$ be the constant value on $H_s$.
Then we have
\begin{align*}
|H_s| \cdot N
&=\sum_{z\in H_s}|\{(x,y)\in(C_i\cap(G\setminus B))\times(C_j\cap(G\setminus B))\mid xy=z\}|\\
&=
|\{(x,y,z)\in(C_i\cap(G\setminus B))\times(C_j\cap(G\setminus B))\times H_s\mid xy=z\}|.
\end{align*}
Since $H_s\subseteq B$, the conditions $x\notin B$ and $z\in H_s$ imply $x^{-1}z\notin B$.
Since $y=x^{-1}z$ is uniquely determined by $x$ and $z$, counting the same triples first by $x$ shows that the last expression is
\[
\sum_{x\in C_i\cap(G\setminus B)}|\{z\in H_s\mid x^{-1}z\in C_j\}|.
\]
For fixed $x\in C_i\cap(G\setminus B)$, the element $x^{-1}$ also lies outside $B$, so by \cref{lem:psl_mixed_counts}, the inner count is unchanged if $H_s$ is replaced by $H_{s^k}$.
Since $|H_s|=|H_{s^k}|=q$, the constant $N$ is the same after replacing $s$ by $s^k$.
\end{proof}

\begin{proof}[Proof of \cref{thm:psl_star_switching}: algebraic isomorphism]
The family $\calD_{q,k}$ is a partition of $G$, contains the identity as a singleton, and is closed under taking inverses.
This follows from $(H_s)^{-1}=H_{s^{-1}}$.

Fix conjugacy classes $C_i,C_j,C_h \in \mathcal{C}_G$, and let $D_i,D_j,D_h$ be the corresponding switched parts in $\calD_{q,k}$.
For $z\in D_h$, put
\[
  N_{ij}(z)=|\{(x,y)\in D_i\times D_j\mid xy=z\}|.
\]
We show that this is independent of the choice of $z\in D_h$, and that it is equal to the structure constant $c_{C_iC_j}^{C_h}$ for $\mathcal{C}_G$.
The switching is the identity outside $B$, so
\[
  D_h\cap(G\setminus B)=C_h\cap(G\setminus B).
\]
On the other hand, $D_h\cap B$ is unchanged for the identity class, the unipotent classes, and the non-split classes; only in the split semisimple case is $H_s\cup H_{s^{-1}}$ replaced by $H_{s^k}\cup H_{s^{-k}}$.

First consider $z\in D_h\cap(G\setminus B)$.
Then $z\in C_h\cap(G\setminus B)$.
Since $B$ is a subgroup, if $xy=z\notin B$, then only the following three cases are possible:
\begin{enumerate}[label=($B^\complement$\arabic*)]
  \item $x\in D_i\cap(G\setminus B),\ y\in D_j\cap(G\setminus B)$; \label{item:psl_outside_outside}
  \item $x\in D_i\cap B,\ y\in D_j\cap(G\setminus B)$; \label{item:psl_borel_outside}
  \item $x\in D_i\cap(G\setminus B),\ y\in D_j\cap B$. \label{item:psl_outside_borel}
\end{enumerate}

The contribution from \ref{item:psl_outside_outside} is
\[
  |\{(x,y)\in(D_i\cap(G\setminus B))\times(D_j\cap(G\setminus B))\mid xy=z\}|
  =
  |\{(x,y)\in(C_i\cap(G\setminus B))\times(C_j\cap(G\setminus B))\mid xy=z\}|,
\]
which is exactly the same before and after switching.

Since $b\in B$ and $z\in G\setminus B$ imply $z^{-1}b\in G\setminus B$, the contribution from \ref{item:psl_borel_outside} is
\[
  |\{(b,y)\in(D_i\cap B)\times(D_j\cap(G\setminus B))\mid by=z\}|
  =
  |\{b\in D_i\cap B\mid z^{-1}b\in C_j^{-1}\}|.
\]
If $D_i$ comes from the identity class, a unipotent class, or a non-split class, then $D_i\cap B=C_i\cap B$, so this agrees with the corresponding contribution before switching.
If $D_i$ comes from a split semisimple class, we look at the sets $H_s$ in $D_i\cap B$ one at a time.
Since a split semisimple class has $s\ne1$, we may apply the first assertion of \cref{lem:psl_mixed_counts}, taking $g=z^{-1}$ and $C=C_j^{-1}$, obtaining
\[
  |\{b\in H_s\mid z^{-1}b\in C_j^{-1}\}|
  =
  |\{b\in H_{s^k}\mid z^{-1}b\in C_j^{-1}\}|.
\]
The case \ref{item:psl_outside_borel} is identical to \ref{item:psl_borel_outside}.

Thus, for $z\in D_h\cap(G\setminus B)$, all three contributions agree with the corresponding contributions before switching.
Hence $N_{ij}(z)$ is equal to the corresponding structure constant for $\mathcal{C}_G$.

Next suppose that $z\in D_h\cap B$.
If $D_h$ comes from a non-split semisimple class, then $D_h\cap B=\emptyset$, so there is nothing to consider.
If $xy=z\in B$, then only the following two cases can occur:
\begin{enumerate}[label=($B$\arabic*)]
  \item $x\in D_i\cap B,\ y\in D_j\cap B$; \label{item:psl_borel_borel}
  \item $x\in D_i\cap(G\setminus B),\ y\in D_j\cap(G\setminus B)$. \label{item:psl_outside_outside2}
\end{enumerate}

We compute the contribution
\begin{equation}
  \label{eq:psl_borel_borel_contribution}
  |\{(x,y)\in(D_i\cap B)\times(D_j\cap B)\mid xy=z\}|
\end{equation}
directly.
If either $D_i$ or $D_j$ comes from a non-split semisimple class, then it does not meet $B$, and so it does not contribute to \eqref{eq:psl_borel_borel_contribution}.

The intersections with $B$ of the identity class and the unipotent classes are not moved by the switching, and are subsets of $H_1$.
For the switched part corresponding to a split semisimple class $C_{\bar r}$, we have
\[
  D_{\bar r}\cap B=H_{r^k}\cup H_{r^{-k}}.
\]

First suppose that both $D_i\cap B$ and $D_j\cap B$ are subsets of $H_1$.
Then their product lies in $H_1$, so the contribution is $0$ if $z\notin H_1$.
If $z\in H_1$, then both sets and the target are unchanged by switching, and the contribution is the same as before switching.

Next suppose that exactly one factor is a subset $A$ of $H_1$, and the other comes from a split part.
For instance, suppose that $D_i$ comes from the identity class or from a unipotent class, and that $D_j$ comes from a split semisimple class.
Then
\[
  D_i\cap B=A\subseteq H_1,\qquad
  D_j\cap B=H_{r^k}\cup H_{r^{-k}}.
\]
Put $V=\{r^k,r^{-k}\}$.
If $z\in H_u$, then applying \cref{lem:psl_borel_counts} with $s=1$ and $t=v$, for each $v\in V$, gives
\[
  |\{(x,y)\in A\times (H_{r^k}\cup H_{r^{-k}})\mid xy=z\}|
  =
  \begin{cases}
    |A|,& u\in V,\\
    0,& u\notin V.
  \end{cases}
\]
Write $u=u_0^k$ for the unique $u_0\in\Sq(\F_q^\times)$.
Thus the switched condition $u\in\{r^k,r^{-k}\}$ is equivalent to the original condition $u_0\in\{r,r^{-1}\}$, and hence this contribution agrees with the original contribution.
The case with $D_i\cap B=H_{r^k}\cup H_{r^{-k}}$ and $D_j\cap B=A\subseteq H_1$ is the same.

It remains to consider the case where both $D_i$ and $D_j$ come from split semisimple classes.
Write
\[
  D_i\cap B=\bigcup_{v\in S_i^k}H_v,\qquad
  D_j\cap B=\bigcup_{w\in S_j^k}H_w,
\]
where, for some $r,t\in \Sq(\F_q^\times)\setminus\{1\}$,
\[
  S_i=\{r,r^{-1}\},\qquad S_j=\{t,t^{-1}\},\qquad
  S_i^k=\{v^k\mid v\in S_i\},\quad
  S_j^k=\{w^k\mid w\in S_j\}.
\]
Let $z\in H_u$.
Applying \cref{lem:psl_borel_counts} with $A=H_v$ and $t=w$, for each $v\in S_i^k$ and $w\in S_j^k$, gives
\[
  |\{(x,y)\in H_v\times H_w\mid xy=z\}|
  =
  \begin{cases}
    q,& vw=u,\\
    0,& vw\ne u.
  \end{cases}
\]
Therefore the contribution is
\[
  q\cdot|\{(v,w)\in S_i^k\times S_j^k\mid vw=u\}|.
\]
Since the map $\Sq(\F_q^\times)\to \Sq(\F_q^\times)$, $v\mapsto v^k$, is a group automorphism, this number of solutions is equal to the corresponding number before switching,
\[
  |\{(v,w)\in S_i\times S_j\mid vw=u_0\}|,
\]
where $u=u_0^k$; if $u=1$, then $u_0=1$.
Thus the contribution from \ref{item:psl_borel_borel} also agrees with the corresponding contribution before switching.

The contribution from \ref{item:psl_outside_outside2} is
\[
  F(z)=|\{(x,y)\in(D_i\cap(G\setminus B))\times(D_j\cap(G\setminus B))\mid xy=z\}|
  =
  |\{(x,y)\in(C_i\cap(G\setminus B))\times(C_j\cap(G\setminus B))\mid xy=z\}|.
\]
If $z$ belongs to the identity class or to a unipotent class, then $F(z)$ is exactly the same as before switching.
If $z$ belongs to a split semisimple class, then by \cref{lem:psl_outside_outside}, $F$ is constant on $H_s$, and this constant is equal to the constant on $H_{s^k}$.
Thus the contribution is preserved in this case as well.

Consequently, for every $z\in D_h$, the number $N_{ij}(z)$ is independent of the choice of $z$, and its value is equal to the corresponding structure constant for $\mathcal{C}_G$.
Therefore all structure constants are preserved, and $\calD_{q,k}$ is a Schur partition algebraically isomorphic to $\mathcal{C}_G$.
\end{proof}

\subsection{Combinatorial non-isomorphism}

In this subsection, we prove the second assertion of \cref{thm:psl_star_switching}; that is, we show that, if $k\not\equiv\pm1\pmod m$, then the Cayley association scheme associated with $\calD_{q,k}$ is not combinatorially isomorphic to $\mathfrak X(G)$.
Let $\calU$ be the set of all nonidentity unipotent elements of $G$, that is, the union of the nonidentity unipotent conjugacy classes.
If $q$ is odd, this is the union of two conjugacy classes; if $q$ is even, it is one conjugacy class.

We record a basic fact about the Sylow $p$-subgroups of $G$.
As above, $H_1$ is a Sylow $p$-subgroup of $G$, and has size $q$.
By Sylow's theorem, every Sylow $p$-subgroup of $G$ of order $q$ is conjugate to $H_1$.
Since unipotence is preserved under conjugation, every nonidentity element of a Sylow $p$-subgroup is a nonidentity unipotent element.
Conversely, every nonidentity unipotent element is a $p$-element, and therefore lies in some Sylow $p$-subgroup.

\begin{lemma}
\label{lem:psl_unipotent_cliques}
Let $q\ge5$, and let $\Delta$ be the Cayley graph on $G$ in which two vertices $a,b$ are adjacent if $ba^{-1}\in\calU$.
Let $K$ be a subset of $G$ of size $q$.
Then the induced subgraph on $K$ is a clique in $\Delta$ if and only if $K$ is a left coset of a Sylow $p$-subgroup of order $q$.
\end{lemma}

\begin{proof}
A left coset of a Sylow $p$-subgroup is a clique of size $q$.
Indeed, let $P$ be such a Sylow $p$-subgroup and write the left coset as $gP$.
Take two distinct elements $gp_1,gp_2\in gP$, with $p_1,p_2\in P$ and $p_1\ne p_2$.
Then $p_2p_1^{-1}\in P\setminus\{1\}$, and since $P$ is a Sylow $p$-subgroup of order $q$, the element $p_2p_1^{-1}$ is nonidentity unipotent.
Thus $(gp_2)(gp_1)^{-1}=g(p_2p_1^{-1})g^{-1}\in\calU$.
Hence any two distinct elements of $gP$ are adjacent in $\Delta$, so $gP$ is a clique.

Conversely, after left translating a clique of size $q$, we may assume that it contains the identity.
Call this clique $K$.
Every nonidentity element of $K$ is a nonidentity unipotent element, and the quotient of any two distinct elements of $K$ is also unipotent.
For $\alpha\in\Pone(\F_q)$, let $P_\alpha$ be the Sylow subgroup fixing $\alpha$.
We show that $K=P_\alpha$ for some $\alpha$.

First suppose that $q$ is odd.
Elements $u_{\infty,\xi}\in P_\infty$ and $u_{\alpha,\eta}\in P_\alpha$ have lifts
\[
  \hat u_{\infty,\xi}=\begin{pmatrix}1&\xi\\0&1\end{pmatrix},
  \qquad
  \hat u_{\alpha,\eta}=
  \begin{pmatrix}
  1-\alpha\eta&\alpha^2\eta\\
  -\eta&1+\alpha\eta
  \end{pmatrix},
\]
respectively.
We call $\xi$ and $\eta$ the parameters of $u_{\infty,\xi}$ and $u_{\alpha,\eta}$.
Consider adjacency in $\Delta$ between $u_{\infty,\xi}$, with $\xi\ne0$, and $u_{\alpha,\eta}$, with $\eta\ne0$.
A direct computation gives
\[
  \Tr(\hat u_{\infty,\xi}^{-1}\hat u_{\alpha,\eta})=2+\xi\eta,
  \qquad
  \Tr(\hat u_{\alpha,\eta}^{-1}\hat u_{\beta,\theta})=2+\eta\theta(\alpha-\beta)^2.
\]
If $\xi,\eta,\theta\ne0$ and $\alpha\ne\beta$, these traces are different from $2$.
Therefore, for these elements to be adjacent, it is necessary that
\begin{equation}
\label{eq:psl_unipotent_clique_trace}
  \xi\eta=-4,
  \qquad
  \eta\theta(\alpha-\beta)^2=-4.
\end{equation}
Suppose now that $u_1,u_2,u_3\in K$, with $u_1,u_2\in P_\alpha$ and $u_3\in P_\beta$, where $\alpha,\beta\in\Pone(\F_q)$ and $\alpha\ne \beta$.
By \eqref{eq:psl_unipotent_clique_trace}, the parameters of $u_1$ and $u_2$, as seen from $u_3$, must be equal.
Thus, if $K$ meets more than one fixed-point Sylow subgroup, then each fixed-point Sylow subgroup meets $K$ in at most one nonidentity element.
Without loss of generality, assume that one such fixed-point Sylow subgroup is $P_\infty$.
Write
\[
  K=\{e,u_{\infty,\xi},u_1,u_2,\ldots,u_{q-2}\},
\]
where $u_{\infty,\xi}\in P_\infty$, and $u_1,u_2,\ldots,u_{q-2}$ lie in distinct fixed-point Sylow subgroups $P_{\alpha_1},P_{\alpha_2},\ldots,P_{\alpha_{q-2}}$, with $\alpha_1,\ldots,\alpha_{q-2}\in\F_q$.
By \eqref{eq:psl_unipotent_clique_trace}, the parameters of $u_1,u_2,\ldots,u_{q-2}$ are all equal.
Applying \eqref{eq:psl_unipotent_clique_trace} again, the values $(\alpha_k-\alpha_l)^2$, for $1\le k<l\le q-2$, are all equal.
Such a subset $\{\alpha_1,\alpha_2,\ldots,\alpha_{q-2}\}$ of $\F_q$ has size at most $2$ if the characteristic is not $3$, and at most $3$ if the characteristic is $3$.
Indeed, after translating and scaling three distinct elements to $0,1,t$, the condition gives $t^2=1$ and $(t-1)^2=1$.
Thus $t=-1$ and $t=2$, which can occur only in characteristic $3$; a fourth element cannot satisfy the same pairwise condition with $0,1,-1$.
Thus the clique has size at most $4$ in characteristic not $3$, and at most $5$ in characteristic $3$.
These bounds are smaller than $q$: in characteristic different from $3$ this follows from $q\ge5$, while in characteristic $3$ the assumption $q\ge5$ gives $q\ge9$.

In characteristic $2$, the same trace computation shows that if $K$ meets more than one fixed-point Sylow subgroup, with one of them taken to be $P_\infty$, then $\xi\eta=0$ is forced.
This is impossible for nonidentity unipotent elements.
Therefore all nonidentity elements of $K$ have the same fixed point $\alpha$.
Hence $K\subseteq P_\alpha$, and since $|K|=|P_\alpha|=q$, we have $K=P_\alpha$.
This proves the lemma.
\end{proof}

Assume from now on that $k\not\equiv\pm1\pmod m$.
This hypothesis implies $m\notin\{1,2,3,4,6\}$, hence $q\ge8$ in even characteristic and $q\ge11$ in odd characteristic; in particular \cref{lem:psl_unipotent_cliques} applies.
Choose $s\in \Sq(\F_q^\times)\setminus\{1\}$ such that
$s_0=s^k\notin\{s,s^{-1}\}$.
Such an $s$ exists: if $s^k\in\{s,s^{-1}\}$ for every $s$, then applying this to a generator of the cyclic group $\Sq(\F_q^\times)$ would give $k\equiv\pm1\pmod m$.
Let $D_{\bar s}$ be the switched part in \eqref{eq:psl-switched-part}.
The set $D_{\bar s}\setminus B=C_{\bar s}\setminus B$ is nonempty: one can conjugate a diagonal representative of $C_{\bar s}$ so that its two fixed points in $\Pone(\F_q)$ do not include $\infty$.
Let $\Gamma_{\bar s}:=\Gamma_{\calU}(D_{\bar s})$, equivalently the induced subgraph of $\Delta$ on $D_{\bar s}$.

\begin{lemma}
\label{lem:psl_star_clique}
The graph $\Gamma_{\bar s}$ has a vertex $x$ contained in a clique of size $q$.
\end{lemma}

\begin{proof}
By the definition of $D_{\bar s}$, we have $H_{s_0}\subseteq D_{\bar s}$, and $H_{s_0}$ is a clique of size $q$ in $\Gamma_{\bar s}$.
Indeed, for distinct $h_{s_0,a},h_{s_0,b}\in H_{s_0}$, the quotient $h_{s_0,b}h_{s_0,a}^{-1}$ lies in $H_1\setminus\{1\}$, and is therefore nonidentity unipotent.
Thus we may take $x=h_{s_0,0}\in H_{s_0}$.
\end{proof}

\begin{lemma}
\label{lem:psl_outside_no_clique}
If $x\in D_{\bar s}\setminus B=C_{\bar s}\setminus B$, then $x$ is not contained in any clique of size $q$ in $\Gamma_{\bar s}$.
\end{lemma}

\begin{proof}
Suppose that $x$ is contained in a clique of size $q$ in $\Gamma_{\bar s}$.
By \cref{lem:psl_unipotent_cliques}, this clique is $xP_\alpha$ for some $\alpha\in\Pone(\F_q)$.
Therefore
\begin{equation}
\label{eq:psl-pencil-contained}
  xP_\alpha\subseteq D_{\bar s}.
\end{equation}

First suppose that $\alpha$ is not a fixed point of $x$.
We keep the original coordinate in which $B=G_\infty$.
If $\alpha=\infty$, take lifts
\[
  \hat x=\begin{pmatrix}a&b\\ c&d\end{pmatrix},
  \qquad c\ne0,
  \qquad
  \hat u_{\infty,t}=\begin{pmatrix}1&t\\0&1\end{pmatrix}.
\]
Then $x u_{\infty,t}\notin B$ for every $t$, and
$\Tr(\hat x\hat u_{\infty,t})=a+d+ct$
is a nonconstant affine function of $t$.
If $\alpha\in\F_q$, write
\[
  \hat u_{\alpha,t}=
  \begin{pmatrix}
  1-\alpha t&\alpha^2t\\
  -t&1+\alpha t
  \end{pmatrix}.
\]
The lower-left entry of $\hat x\hat u_{\alpha,t}$ is
$c-t(c\alpha+d)$,
so $x u_{\alpha,t}\notin B$ for all but at most one value of $t$.
Moreover
\[
  \Tr(\hat x\hat u_{\alpha,t})
  =a+d+t(c\alpha^2+(d-a)\alpha-b).
\]
The coefficient of $t$ is nonzero precisely because $\alpha$ is not a fixed point of $x$.
Thus, in both cases, there are more than two values of $t$ for which $x u_{\alpha,t}\notin B$ and the trace is a nonconstant affine function of $t$.
For those values, \eqref{eq:psl-pencil-contained} implies $x u_{\alpha,t}\in C_{\bar s}\setminus B$.
In odd characteristic, the value of $\Tr(\hat x\hat u_{\alpha,t})^2$ would then have to be the fixed value attached to $\bar s$, although a nonconstant quadratic polynomial cannot take one fixed value at more than two points.
In even characteristic, the trace itself would have to be the fixed value attached to $\bar s$, although a nonconstant affine function cannot take one fixed value at more than one point.
This contradiction proves that \eqref{eq:psl-pencil-contained} is impossible in this case.

It remains to consider the case in which $\alpha$ is a fixed point of $x$.
Since $x\notin B$, we have $x\cdot\infty\ne\infty$, and since $\alpha$ is a fixed point of $x$, we also have $x^{-1}\cdot\infty\ne\alpha$.
The group $P_\alpha$ acts transitively on $\Pone(\F_q)\setminus\{\alpha\}$, so there exists $p\in P_\alpha$ with
$p\cdot\infty=x^{-1}\cdot\infty$.
It follows that
$(xp)\cdot\infty
=x\cdot(p\cdot\infty)
=x\cdot(x^{-1}\cdot\infty)
=\infty$,
and hence $xp\in B$.
Moreover, since $p$ is a unipotent element fixing $\alpha$, in coordinates based at $\alpha$ we may choose a lift $\hat x$ of $x$ which is upper triangular, and a lift $\hat p$ of $p$ which is upper triangular with both diagonal entries equal to $1$.
Thus the ratio of the eigenvalues of $\hat x\hat p$ is the same as that of $\hat x$, and $xp\in C_{\bar s}$.
Hence
\[
  xp\in C_{\bar s}\cap B=H_s\cup H_{s^{-1}}.
\]
On the other hand, $D_{\bar s}$ is disjoint from $H_s\cup H_{s^{-1}}$, so $xp\notin D_{\bar s}$, contradicting \eqref{eq:psl-pencil-contained}.
Therefore $x$ is not contained in any clique of size $q$ in $\Gamma_{\bar s}$.
\end{proof}

\begin{proof}[Proof of \cref{thm:psl_star_switching}: non-isomorphism]
By \cref{lem:psl_star_clique}, the local graph $\Gamma_{\bar s}$ has a vertex contained in a clique of size $q$.
By \cref{lem:psl_outside_no_clique}, the same graph also has a vertex which is not contained in any clique of size $q$.
Thus $\Gamma_{\bar s}$ is not vertex-transitive.

Suppose that the switched association scheme were combinatorially isomorphic to the group association scheme $\mathfrak X(G)$.
Under this isomorphism, $D_{\bar s}$ would correspond to a conjugacy class $C$ of $G$, and $\calU$ would correspond to a union $Y$ of conjugacy classes of $G$.
By \cref{lem:inv}, $\Gamma_{\bar s}=\Gamma_{\cal U}(D_{\bar s},e)$ would be isomorphic to $\Gamma_Y(C,f(e))$.
By \cref{lem:local-basepoint}, $\Gamma_Y(C,f(e))\cong\Gamma_Y(C,e)$.
The latter graph is vertex-transitive under the conjugation action of $G$.
This is a contradiction.
Therefore the switched association scheme is not combinatorially isomorphic to $\mathfrak X(G)$.
\end{proof}

The two preceding subsections together prove \cref{thm:psl_star_switching}.
Since $\Gamma_{\bar s}$ is not vertex-transitive, we also obtain the following corollary.

\begin{corollary}
\label{cor:psl_switched_nonschurian}
Assume that $k\not\equiv\pm1\pmod m$.
Then the Cayley association scheme associated with $\calD_{q,k}$ is non-Schurian.
\end{corollary}

\begin{proof}
Since $k\not\equiv\pm1\pmod m$, choose, as above, $s\in \Sq(\F_q^\times)\setminus\{1\}$ such that
\[
  s_0=s^k\notin\{s,s^{-1}\},
\]
and consider the local graph $\Gamma_{\bar s}$ on $D_{\bar s}$.
Suppose, for a contradiction, that the Cayley association scheme associated with $\calD_{q,k}$ is Schurian.
Since $\Gamma_{\bar s}$ is the local graph on the basic set $D_{\bar s}$, with adjacency defined by the union $\calU$ of the unipotent basic sets, \cref{lem:schurian-local-transitive} implies that $\Gamma_{\bar s}$ is vertex-transitive.
But by \cref{lem:psl_star_clique,lem:psl_outside_no_clique}, the graph $\Gamma_{\bar s}$ is not vertex-transitive.
This contradiction shows that the Cayley association scheme associated with $\calD_{q,k}$ is non-Schurian.
\end{proof}

\section{Twisted conjugacy classes}
\label{sec:twisted_conjugacy}

In this section, we give a twisting construction for finite groups that is different from the one used in \cref{sec:psl_star_switching}.
The Schur partition obtained from this construction has the same multiplication table as the conjugacy classes.
In \cref{sec:psl_star_switching}, we change each split semisimple conjugacy class only inside a fixed Borel subgroup.
Here, we instead replace two conjugacy classes by a new partition of their union.

Let $G$ be a finite group, and let $\mathcal{C}_G=\{C_i\}_{i\in\mathcal{I}}$ be the family of its conjugacy classes.
Assume that $i_1,i_2\in\mathcal{I}$ satisfy the following conditions:
\begin{enumerate}[label=(C\arabic*)]
  \item \label{item:lem:main1} $|C_{i_1}|=|C_{i_2}|$;
  \item \label{item:lem:main2} $\{C_{i_1}^{-1},C_{i_2}^{-1}\}=\{C_{i_1},C_{i_2}\}$;
  \item \label{item:lem:main3} If $U=\underline{C_{i_1}}+\underline{C_{i_2}}$ and $S=\underline{C_{i_1}}-\underline{C_{i_2}}$, then for every $j\in\mathcal{I}\setminus\{i_1,i_2\}$ we have $\underline{C_j}S=\lambda_jS$; in other words, $S$ is a simultaneous eigenvector for these class sums.
  \item \label{item:lem:main4} $S^2$ can be expressed as a linear combination of $\{\underline{C_i}\}_{i\in\mathcal{I}\setminus\{i_1,i_2\}}\cup\{U\}$.
\end{enumerate}

\begin{lemma}\label{lem:main}
Assume that a finite group $G$ has $i_1,i_2\in\mathcal{I}$ satisfying $i_1\neq i_2$ and \ref{item:lem:main1}--\ref{item:lem:main4}.
Let $D_1,D_2$ be a partition of $C_{i_1}\cup C_{i_2}$ satisfying
\begin{enumerate}[label=(D\arabic*)]
  \item \label{item:lem:main1D} $|D_1|=|D_2|$;
  \item \label{item:lem:main2D} $\{D_1^{-1},D_2^{-1}\}=\{D_1,D_2\}$;
  \item \label{item:lem:main3D} if $T=\underline{D_1}-\underline{D_2}$, then for every $j\in\mathcal{I}\setminus\{i_1,i_2\}$ we have $\underline{C_j}T=\lambda_jT$;
  \item \label{item:lem:main4D} $T^2=S^2$.
\end{enumerate}
Then
$\mathcal{S}'=\{C_i\}_{i\in\mathcal{I}\setminus\{i_1,i_2\}}\cup\{D_1,D_2\}$
is a Schur partition, and $\mathcal{S}'$ is algebraically isomorphic to $\mathcal{C}_G$.
\end{lemma}
\begin{proof}
  First define the vector subspace of $\mathbb{C}G$
  \[
    \mathcal{A}_0:=\operatorname{Span}_{\mathbb{C}}
      \bigl(\{\underline{C_i}\mid i\in\mathcal{I}\setminus\{i_1,i_2\}\}\cup\{U\}\bigr).
  \]
  Then, as vector spaces, we have
  $\mathcal{A}(\mathcal{C}_G)=\mathcal{A}_0\oplus \mathbb{C}S$.

  \medskip
  \noindent
  \textbf{Step 1: Each element of $\mathcal{A}_0$ acts on $S$ and on $T$ by the same scalar.}
  By assumption \ref{item:lem:main3}, for $j\in\mathcal{I}\setminus\{i_1,i_2\}$ we have
  $\underline{C_j}S=\lambda_jS$.
  Using the full group sum
  $\underline{G}:=\sum_{g\in G}g$, we have
  $\underline{G}\,\underline{X}=|X|\,\underline{G}$
  for every subset $X\subseteq G$.
  By assumption \ref{item:lem:main1}, $|C_{i_1}|=|C_{i_2}|$, and hence
  \[
  \underline{G}\,S
  =(|C_{i_1}|-|C_{i_2}|)\underline{G}
  =0.
  \]
  On the other hand,
  \[
  \underline{G}
  =\sum_{j\in\mathcal{I}\setminus\{i_1,i_2\}}\underline{C_j}
  +U.
  \]
  Therefore
  \[
    US
    =-\sum_{j\in\mathcal{I}\setminus\{i_1,i_2\}}\underline{C_j}S
    =-\sum_{j\in\mathcal{I}\setminus\{i_1,i_2\}}\lambda_j S.
  \]
  Hence for every $X\in\mathcal{A}_0$ there exists a scalar $\chi(X)\in\mathbb{C}$ such that
  $XS=\chi(X)S$.
  Similarly, from assumptions \ref{item:lem:main1D} and \ref{item:lem:main3D},
  $\underline{G}\,T=(|D_1|-|D_2|)\underline{G}=0$, and
  \[
    UT
    =-\sum_{j\in\mathcal{I}\setminus\{i_1,i_2\}}\underline{C_j}T
    =-\sum_{j\in\mathcal{I}\setminus\{i_1,i_2\}}\lambda_j T.
  \]
  Thus the same linear functional $\chi:\mathcal{A}_0\to\mathbb{C}$ satisfies
  $XT=\chi(X)T$ for all $X\in\mathcal{A}_0$.

  \medskip
  \noindent
  \textbf{Step 2: $\mathcal{A}_0$ is a subalgebra.}
  Take $X,Y\in\mathcal{A}_0$.
  Since $\mathcal{A}(\mathcal{C}_G)=\mathcal{A}_0\oplus\mathbb{C}S$, there exist $Z\in\mathcal{A}_0$ and $\alpha\in\mathbb{C}$ such that
  $XY=Z+\alpha S$.
  Multiplying this equality on the right by $S$, we obtain
  $\chi(X)\chi(Y)S=\chi(Z)S+\alpha S^2$.
  By assumption \ref{item:lem:main4}, $\alpha S^2\in\mathcal{A}_0$.
  Viewing this equality with respect to the decomposition $\mathcal{A}_0\oplus\mathbb{C}S$, the left-hand side has no $\mathcal{A}_0$-component, and hence $\alpha S^2=0$.
  By \ref{item:lem:main2}, the coefficient of the identity element in $S^2$ is either $2|C_{i_1}|$, if $C_{i_1}^{-1}=C_{i_1}$ and $C_{i_2}^{-1}=C_{i_2}$, or $-2|C_{i_1}|$, if $C_{i_1}^{-1}=C_{i_2}$.
  In particular $S^2\neq0$.
  Therefore $\alpha=0$, and hence $XY=Z\in\mathcal{A}_0$.
  Thus $\mathcal{A}_0$ is a subalgebra.

  \medskip
  \noindent
  \textbf{Step 3: $\mathcal{S}'$ is a Schur partition.}
  Since $D_1\sqcup D_2=C_{i_1}\cup C_{i_2}$, together with the other $C_i$ these sets form a partition of $G$.
  The singleton $\{e\}$ is also a part of $\mathcal{S}'$.
  Indeed, this is clear unless one of $C_{i_1},C_{i_2}$ is the identity class; in that case, assumptions \ref{item:lem:main1} and \ref{item:lem:main1D} imply that $D_1$ and $D_2$ are both singletons.
  By \ref{item:lem:main2D}, $\{D_1^{-1},D_2^{-1}\}=\{D_1,D_2\}$, so $\mathcal{S}'$ is closed under inverses.
  Also, since $\underline{D_1}=(U+T)/2$ and $\underline{D_2}=(U-T)/2$, we have
  \[
  \mathcal{A}(\mathcal{S}')=\mathcal{A}_0\oplus \mathbb{C}T.
  \]
  By Step 1, each element of $\mathcal{A}_0$ acts by a scalar on $T$, and by assumption \ref{item:lem:main4D} we have $T^2=S^2\in\mathcal{A}_0$.
  Hence $\mathcal{A}(\mathcal{S}')$ is a subalgebra of $\mathbb{C}G$, and it is spanned by $\{\underline{X} \mid X\in\mathcal{S}'\}$.
  Therefore $\mathcal{S}'$ is a Schur partition.

  \medskip
  \noindent
  \textbf{Step 4: Equality of multiplication tables and algebraic isomorphism.}
  Consider the bases
  \[
    \mathcal{B}
    :=\{\underline{C_i}\mid i\in\mathcal{I}\setminus\{i_1,i_2\}\}\cup\{U,S\},
    \qquad
    \mathcal{B}'
    :=\{\underline{C_i}\mid i\in\mathcal{I}\setminus\{i_1,i_2\}\}\cup\{U,T\}.
  \]
  By Step 2, products inside $\mathcal{A}_0$ are the same in both algebras.
  By Step 1, products of elements of $\mathcal{A}_0$ with $S$ and with $T$ are described by the same functional $\chi$.
  Finally, by assumption \ref{item:lem:main4D}, $S^2=T^2$.
  Therefore the multiplication tables with respect to $\mathcal{B}$ and $\mathcal{B}'$ coincide completely.
  Thus the linear map
  \[
    \Psi:\mathcal{A}(\mathcal{C}_G)\longrightarrow \mathcal{A}(\mathcal{S}'),
    \qquad
    \Psi|_{\mathcal{A}_0}=\operatorname{id},\quad \Psi(S)=T
  \]
  is an algebra isomorphism.
  Finally,
  \[
    \Psi(\underline{C_{i_1}})
    =\Psi\Bigl(\frac{U+S}{2}\Bigr)
    =\frac{U+T}{2}
    =\underline{D_1},
    \qquad
    \Psi(\underline{C_{i_2}})
    =\Psi\Bigl(\frac{U-S}{2}\Bigr)
    =\frac{U-T}{2}
    =\underline{D_2},
  \]
  and all other basic sets are fixed.
  Hence the correspondence between basic sets
  \[
    C_i\mapsto C_i\ (i\neq i_1,i_2),\qquad
    C_{i_1}\mapsto D_1,
    \qquad
    C_{i_2}\mapsto D_2
  \]
  preserves structure constants.
  In other words, it gives an algebraic isomorphism of S-rings.
\end{proof}

\begin{remark}
Assumptions \ref{item:lem:main1}--\ref{item:lem:main4} include the assertion that the partition obtained by replacing $C_{i_1}$ and $C_{i_2}$ with their union,
\[
  \{C_i\mid i\in\mathcal{I}\setminus\{i_1,i_2\}\}
  \cup\{C_{i_1}\cup C_{i_2}\},
\]
gives a fusion association scheme.
Indeed, as shown in Step 2 of the proof,
\[
  \mathcal{A}_0
  =\operatorname{Span}_{\mathbb{C}}\Bigl(
      \{\underline{C_i}\mid i\in\mathcal{I}\setminus\{i_1,i_2\}\}
      \cup\{\underline{C_{i_1}}+\underline{C_{i_2}}\}
    \Bigr)
\]
is a subalgebra of the original Bose--Mesner algebra.
\end{remark}

The original partition $D_1=C_{i_1}$, $D_2=C_{i_2}$ satisfies the hypotheses of \cref{lem:main} whenever $C_{i_1},C_{i_2}$ satisfy \ref{item:lem:main1}--\ref{item:lem:main4}.
In order to estimate how many partitions $D_1,D_2$ other than $C_{i_1},C_{i_2}$ may satisfy the hypotheses of \cref{lem:main}, we consider the following graphs.

For $C_{i_1},C_{i_2}$ satisfying \ref{item:lem:main1}--\ref{item:lem:main4} and for $j\in\mathcal{I}\setminus\{i_1,i_2\}$, let $V:=C_{i_1}\cup C_{i_2}$.
Let $\Gamma_j$ be the directed graph with vertex set $V$ and arc set $\{(x,y)\in V\times V\mid yx^{-1}\in C_j\}$, and let $\Gamma_V$ be the directed graph with vertex set $V$ and arc set $\{(x,y)\in V\times V\mid yx^{-1}\in V\}$.
Let $A_j$ and $A_V$ denote the adjacency matrices of $\Gamma_j$ and $\Gamma_V$, respectively.
Define the sign vector $\chi_S$ on $V$ by
\[
  \chi_S(x)=
  \begin{cases}
  1 & \text{if } x\in C_{i_1},\\
  -1 & \text{if } x\in C_{i_2}.
  \end{cases}
\]

\begin{lemma}
  \label{lem:main2}
  The matrix $A_j$ has eigenvalue $\lambda_j$ appearing in \ref{item:lem:main3}.
  Moreover, $A_V$ has eigenvalue $\lambda_U$, where
  \[
  \lambda_U:=-\sum_{j\in\mathcal{I}\setminus\{i_1,i_2\}}\lambda_j.
  \]

  Furthermore, if $\{D_1,D_2\}$ is any partition satisfying the hypotheses of \cref{lem:main}, and if $\chi_T$ is the sign vector on $V$ which is $1$ on $D_1$ and $-1$ on $D_2$, then $\chi_T$ belongs to the $\lambda_j$-eigenspace of $A_j$ and to the $\lambda_U$-eigenspace of $A_V$.
\end{lemma}
\begin{proof}
  For a function $\chi$ on $V$, put
  \[
  X_{\chi}:=\sum_{x\in V}\chi(x)x\in \mathbb{C}G.
  \]
  For $x\in V$, the coefficient of $x$ in $\underline{C_j^{-1}}X_{\chi}$ is
  \[
    \sum_{\substack{y\in V\\ yx^{-1}\in C_j}}
     \chi(y)
    =
    \sum_{y\in V}(A_j)_{xy}\chi(y).
  \]
  Therefore, if $\underline{C_j^{-1}}X_{\chi}=\lambda_j X_{\chi}$, then $A_j\chi=\lambda_j\chi$.
  Since $V^{-1}=V$ by \ref{item:lem:main2}, the same coefficient calculation applies to $\Gamma_V$.
  Similarly, if $U X_{\chi}=\lambda_U X_{\chi}$, then $A_V\chi=\lambda_U\chi$.

  By \ref{item:lem:main2}, the inverse anti-automorphism of $\mathbb C G$ sends $S$ to $\varepsilon S$ for some $\varepsilon\in\{\pm1\}$.
  Applying this anti-automorphism to $\underline{C_j}S=\lambda_jS$, and using that class sums are central, gives $\underline{C_j^{-1}}S=\lambda_jS$.
  Hence, by \ref{item:lem:main3}, the sign vector $\chi_S$ is a $\lambda_j$-eigenvector of $A_j$.
  Moreover, in the proof of \cref{lem:main} we have also shown $US=\lambda_U S$, and hence $\chi_S$ is a $\lambda_U$-eigenvector of $A_V$.

  The same argument using \ref{item:lem:main2D} and \ref{item:lem:main3D} gives $\underline{C_j^{-1}}T=\lambda_jT$.
  Hence the sign vector $\chi_T$ also belongs to the $\lambda_j$-eigenspace of $A_j$.
  Similarly, since $UT=\lambda_U T$ was also shown in the proof of \cref{lem:main}, the sign vector $\chi_T$ belongs to the $\lambda_U$-eigenspace of $A_V$ as well.
\end{proof}

From \cref{lem:main2}, we obtain the following corollary.
\begin{corollary}
  \label{cor:unique}
  If there exists $j\in\mathcal{I}\setminus\{i_1,i_2\}$ such that the $\lambda_j$-eigenspace of $A_j$ is one-dimensional, or if the $\lambda_U$-eigenspace of $A_V$ is one-dimensional, then every partition $\{D_1,D_2\}$ satisfying the hypotheses of \cref{lem:main} is equal to $\{C_{i_1},C_{i_2}\}$.
\end{corollary}
\begin{proof}
By \cref{lem:main2}, the sign vector $\chi_T$ associated with any such partition belongs to the same relevant eigenspace as the original vector $\chi_S$.
If that eigenspace is one-dimensional, then $\chi_T=a\chi_S$ for some scalar $a$.
Since both vectors take only the values $\pm1$, we have $a=\pm1$.
Hence the two parts of the partition are exactly $C_{i_1}$ and $C_{i_2}$, up to order.
\end{proof}

From now on, $\mathcal{C}^{\tw}_G$ denotes a Schur partition of the form
\[
\mathcal{S}'=\{C_i\}_{i\in\mathcal{I}\setminus\{i_1,i_2\}}\cup\{D_1,D_2\}
\]
that is not combinatorially isomorphic to $\mathcal{C}_G$.

\paragraph{Computational verification.}
The finite computations cited in \cref{sec:X(A_n),sec:other_finite_groups_twist} were run in GAP 4.15.1 \cite{GAP4}.
The scripts are included under 
\begin{center}
\path{https://github.com/Kurihara190/Twisted-primitive-group-association-schemes}    
\end{center}
 and can be run from the repository root, for example as \texttt{gap A6.g}.
Each script writes a corresponding log under \path{out/}; these logs record the true/false checks used in the text.
The additional script \path{checktwistA6.g} gives an independent AssociationSchemes-package check for the $\mathfrak A_6$ twist, with output in \path{out/result_checktwistA6.txt}.
The scripts for the small-group search and the independent checks of Cayley association schemes use the GAP packages SmallGrp and AssociationSchemes, and the larger alternating-group spectral checks use cvec where needed \cite{SmallGrp,AssociationSchemesPackage,cvecPackage}.

\section{Twists for alternating groups}\label{sec:X(A_n)}

In this section, we apply \cref{lem:main} to alternating groups.
For $\mathfrak A_6$ and $\mathfrak A_8$, we obtain Schur partitions that are algebraically isomorphic but not combinatorially isomorphic to the corresponding partitions into conjugacy classes.
For $\mathfrak A_4$, $\mathfrak A_5$, $\mathfrak A_7$, and $\mathfrak A_9$, we show that any partition satisfying the hypotheses of \cref{lem:main} is the original one, up to interchanging its two parts.

\subsection{$\mathfrak{X}(\mathfrak{A}_6)$ and a twist}\label{sec:A6}
Let $\mathfrak{A}_6$ be the alternating group of degree $6$.
The conjugacy classes of $\mathfrak{A}_6$ are
$C_0, C_{(2,2)}, C_{(3)}, C_{(3,3)}, C_{(4,2)}$
and the two split $5$-cycle classes
$C_{(5)_1}, C_{(5)_2}$.
Here the subscript $\lambda$ in $C_\lambda$ denotes the cycle type.
In order to fix labels for the two $5$-cycle classes, we set
\[
  x_0:=(1\,2\,3\,4\,5)\in C_{(5)_1}.
\]
With this convention, one can determine to which split class an arbitrary $5$-cycle belongs as follows.
Let $x=(a_1\,a_2\,a_3\,a_4\,a_5)$, and let $a_6$ be the fixed point of $x$.
Define $\sigma_x\in\mathfrak{S}_6$ by
$\sigma_x(i)=a_i$ ($1\le i\le 6$).
Then $x=\sigma_x x_0\sigma_x^{-1}$, and
\[
  x\in C_{(5)_1}\Longleftrightarrow \sigma_x\in\mathfrak{A}_6,
  \qquad
  x\in C_{(5)_2}\Longleftrightarrow \sigma_x\notin\mathfrak{A}_6.
\]
Indeed, the centralizer of $x_0$ in $\mathfrak{S}_6$ is $\langle x_0\rangle$, all of whose elements are even permutations.
Therefore, if $x=\sigma x_0\sigma^{-1}=\tau x_0\tau^{-1}$, then $\tau^{-1}\sigma\in\langle x_0\rangle$, so $\sigma$ and $\tau$ have the same parity.
Hence the criterion is well-defined.
Thus, for example, we may take $(1\,2\,3\,4\,6)$ as a representative of $C_{(5)_2}$.
The class sizes are
\[
|C_0|=1,
|C_{(2,2)}|=45,
|C_{(3)}|=40,
|C_{(3,3)}|=40,
|C_{(4,2)}|=90,
|C_{(5)_1}|=72,
|C_{(5)_2}|=72.
\]

First, in the original group association scheme $\mathfrak{X}(\mathfrak{A}_6)$, put
\[
U:=\underline{C_{(5)_1}}+\underline{C_{(5)_2}},
\qquad
S:=\underline{C_{(5)_1}}-\underline{C_{(5)_2}}.
\]
The intersection matrices immediately imply
\begin{align}
\underline{C_{(2,2)}}\, S &=0, & \underline{C_{(4,2)}}\, S&=0, \label{eq:s_eigs}\\
\underline{C_{(3)}}\, S &=-5\,S, & \underline{C_{(3,3)}}\, S&=-5\,S, \notag\\
S^2&=144\,\underline{C_0}-18\,\underline{C_{(3)}}-18\,\underline{C_{(3,3)}}+9\,U. \label{eq:s_square}
\end{align}
The intersection matrices and the calculations leading to \eqref{eq:s_eigs} and \eqref{eq:s_square} are included in \path{A6.g}; the corresponding output is \path{out/result_A6.txt}.

We construct $D_1,D_2$ for which analogous equalities hold for $T=\underline{D_1}-\underline{D_2}$.
For a $5$-cycle $x=(a_1\,a_2\,a_3\,a_4\,a_5)$, define the edge set of the undirected cycle of $x$ by
\[
  E(x):=\bigl\{\{a_1,a_2\},\{a_2,a_3\},\{a_3,a_4\},\{a_4,a_5\},\{a_5,a_1\}\bigr\},
\]
and define the number of edges crossing the cut $\{1,2,3\}|\{4,5,6\}$ by
\[
  \operatorname{cr}(x)
  :=\bigl|\{\{u,v\}\in E(x)\mid
  |\{u,v\}\cap\{1,2,3\}|=|\{u,v\}\cap\{4,5,6\}|=1\}\bigr|.
\]
The $5$-element set $\{a_1,a_2,a_3,a_4,a_5\}$ is always divided into a $2$-element subset and a $3$-element subset with respect to the cut $\{1,2,3\}|\{4,5,6\}$.
Thus $\operatorname{cr}(x)\in\{2,4\}$.
Moreover $\operatorname{cr}(x)=\operatorname{cr}(x^{-1})$, and $x^2$ gives the complementary $5$-cycle on the same five points, so
\[
\operatorname{cr}(x^2)=6-\operatorname{cr}(x).
\]
For $r=1,2$, define
\[
C^{(2)}_{(5)_r}=\{x\in C_{(5)_r}\mid \operatorname{cr}(x)=2\},
\quad
C^{(4)}_{(5)_r}=\{x\in C_{(5)_r}\mid \operatorname{cr}(x)=4\}.
\]
Then $|C^{(2)}_{(5)_r}|=|C^{(4)}_{(5)_r}|=36$.
We show this below.
For each $5$-cycle $x$, there exists a unique edge $\{b_1,b_2\}$ of $E(x)$ satisfying the following condition: $\{b_1,b_2\}$ does not cross the cut $\{1,2,3\}|\{4,5,6\}$, while the two edges $\{b_1,c_1\}$ and $\{b_2,c_2\}$ of $E(x)$ adjacent to it both cross this cut.
Let $d$ be the remaining vertex of $E(x)$.
For $x=(c_1,b_1,b_2,c_2,d)$, define
$x'=(b_1,c_1,c_2,b_2,d)$.
Then $\sigma_{x'}=(b_1\,c_1)(b_2\,c_2)\sigma_x$, and therefore $x$ and $x'$ belong to the same conjugacy class.
Moreover, $\operatorname{cr}(x')=6-\operatorname{cr}(x)$.
For example, $x=(3\,4\,5\,2\,1)$ and $x'=(4\,3\,2\,5\,1)$ are mapped to each other by the above correspondence, and the edge sets $E(x)$ and $E(x')$ are as shown in \cref{fig:cr-switch-example}.
Applying the same operation to $x'$ brings us back to the original $x$.
Hence this correspondence is a bijection on each $C_{(5)_r}$ which exchanges the elements with $\operatorname{cr}=2$ and those with $\operatorname{cr}=4$.
Therefore, in each class $C_{(5)_r}$, the number of elements with $\operatorname{cr}=2$ equals the number of elements with $\operatorname{cr}=4$.

\begin{figure}[htbp]
\centering
\begin{tikzpicture}[
  scale=0.9,
  every node/.style={font=\small},
  point/.style={circle, fill=black, inner sep=1.4pt}
]
  \node at (-3.1,2.05) {$x=(3\,4\,5\,2\,1)$};
  \node at ( 3.1,2.05) {$x'=(4\,3\,2\,5\,1)$};
  \draw[<->] (-0.75,2.05) -- (0.75,2.05);

  \begin{scope}[xshift=-3.1cm]
    \draw (0.1,-1.5) -- (0.1,1.5);
    \node at (-1.15,1.35) {$\{1,2,3\}$};
    \node at ( 1.15,1.35) {$\{4,5,6\}$};

    \coordinate (xthree) at (-0.23, 1.07);
    \coordinate (xfour)  at ( 1.05, 0.65);
    \coordinate (xfive)  at ( 1.05,-0.70);
    \coordinate (xtwo)   at (-0.23,-1.12);
    \coordinate (xone)   at (-1.03,-0.03);

    \draw (xthree) -- (xfour) -- (xfive) -- (xtwo) -- (xone) -- (xthree);
    \node[point] at (xone) {};
    \node[left] at (xone) {$1$};
    \node[point] at (xtwo) {};
    \node[below] at (xtwo) {$2$};
    \node[point] at (xthree) {};
    \node[above] at (xthree) {$3$};
    \node[point] at (xfour) {};
    \node[right] at (xfour) {$4$};
    \node[point] at (xfive) {};
    \node[right] at (xfive) {$5$};
  \end{scope}

  \begin{scope}[xshift=3.1cm]
    \draw (0.1,-1.5) -- (0.1,1.5);
    \node at (-1.15,1.35) {$\{1,2,3\}$};
    \node at ( 1.15,1.35) {$\{4,5,6\}$};

    \coordinate (ythree) at (-0.23, 1.07);
    \coordinate (yfour)  at ( 1.05, 0.65);
    \coordinate (yfive)  at ( 1.05,-0.70);
    \coordinate (ytwo)   at (-0.23,-1.12);
    \coordinate (yone)   at (-1.03,-0.03);

    \draw (yfour) -- (ythree) -- (ytwo) -- (yfive) -- (yone) -- (yfour);
    \node[point] at (yone) {};
    \node[left] at (yone) {$1$};
    \node[point] at (ytwo) {};
    \node[below] at (ytwo) {$2$};
    \node[point] at (ythree) {};
    \node[above] at (ythree) {$3$};
    \node[point] at (yfour) {};
    \node[right] at (yfour) {$4$};
    \node[point] at (yfive) {};
    \node[right] at (yfive) {$5$};
  \end{scope}
\end{tikzpicture}
\caption{The correspondence from $x=(3\,4\,5\,2\,1)$ to $x'=(4\,3\,2\,5\,1)$.}
\label{fig:cr-switch-example}
\end{figure}
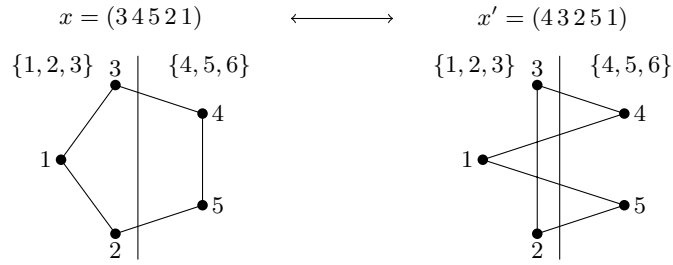

Now put
\begin{align}
D_1&:=C^{(2)}_{(5)_1}\sqcup C^{(4)}_{(5)_2}, \label{eq:D1_rule}\\
D_2&:=C^{(4)}_{(5)_1}\sqcup C^{(2)}_{(5)_2}. \label{eq:D2_rule}
\end{align}
Equivalently, the same construction can be described in terms of Sylow $5$-subgroups.
For any $5$-cycle $g$, it is known that $g$ and $g^2$ belong to different conjugacy classes.
Thus, for example, if $g\in C_{(5)_1}$, then $g^4=g^{-1}\in C_{(5)_1}$ and $g^2,g^3=(g^2)^{-1}\in C_{(5)_2}$.
For a subgroup $H\le \mathfrak A_6$, write $H^\#:=H\setminus\{e\}$.
Then $D_1$ and $D_2$ can be written as
\begin{equation}\label{eq:D12_sylow_union}
  D_1=\bigsqcup_{g\in\mathcal{G}_1}\langle g\rangle^\#,
  \qquad
  D_2=\bigsqcup_{g\in\mathcal{G}_2}\langle g\rangle^\#,
\end{equation}
where $\mathcal{G}_1$ and $\mathcal{G}_2$ consist of the following eighteen $5$-cycles, respectively:
\begin{align*}
\mathcal{G}_1=\{&
(1\,2\,3\,4\,5),\ (1\,2\,3\,5\,6),\ (1\,2\,3\,6\,4),\ (1\,2\,4\,3\,5),\ (1\,2\,4\,5\,3),\ (1\,2\,4\,6\,5),\\
& (1\,2\,5\,3\,6),\ (1\,2\,5\,4\,6),\ (1\,2\,5\,6\,3),\ (1\,2\,6\,3\,4),\ (1\,2\,6\,4\,3),\ (1\,2\,6\,5\,4),\\
& (1\,3\,4\,5\,6),\ (1\,3\,5\,6\,4),\ (1\,3\,6\,4\,5),\ (2\,3\,4\,6\,5),\ (2\,3\,5\,4\,6),\ (2\,3\,6\,5\,4)
\},\\[1ex]
\mathcal{G}_2=\{&
(1\,2\,3\,4\,6),\ (1\,2\,3\,5\,4),\ (1\,2\,3\,6\,5),\ (1\,2\,4\,3\,6),\ (1\,2\,4\,5\,6),\ (1\,2\,4\,6\,3),\\
& (1\,2\,5\,3\,4),\ (1\,2\,5\,4\,3),\ (1\,2\,5\,6\,4),\ (1\,2\,6\,3\,5),\ (1\,2\,6\,4\,5),\ (1\,2\,6\,5\,3),\\
& (1\,3\,4\,6\,5),\ (1\,3\,5\,4\,6),\ (1\,3\,6\,5\,4),\ (2\,3\,4\,5\,6),\ (2\,3\,5\,6\,4),\ (2\,3\,6\,4\,5)
\}.
\end{align*}
Here the elements of $\mathcal{G}_1$ generate pairwise distinct Sylow $5$-subgroups, and the elements of $\mathcal{G}_2$ also generate pairwise distinct Sylow $5$-subgroups.
Moreover, the eighteen Sylow $5$-subgroups obtained from $\mathcal{G}_1$ and the eighteen obtained from $\mathcal{G}_2$ are mutually distinct.
Therefore, for $r=1,2$, the set $D_r$ is the union of the nonidentity elements of eighteen Sylow $5$-subgroups, and $|D_r|=18\cdot4=72$.
In particular, $|D_1\cap C_{(5)_1}|=|D_1\cap C_{(5)_2}|=36$, and this partition is different from the original partition into $C_{(5)_1}$ and $C_{(5)_2}$.

From now on, for the sets $D_1,D_2$ defined above, we write
\[
\mathcal{C}^{\tw}_{\mathfrak{A}_6}=
\{C_0,C_{(2,2)},C_{(3)},C_{(3,3)},C_{(4,2)},D_1,D_2\}.
\]

\begin{theorem}\label{thm:twistedA6}
$\mathcal{C}^{\tw}_{\mathfrak{A}_6}$ is a Schur partition over $\mathfrak{A}_6$.
Moreover, $\mathcal{C}_{\mathfrak{A}_6}$ and $\mathcal{C}^{\tw}_{\mathfrak{A}_6}$ are algebraically isomorphic, but not combinatorially isomorphic.
\end{theorem}

\begin{proof}
First, by \eqref{eq:D12_sylow_union}, we have $D_1^{-1}=D_1$ and $D_2^{-1}=D_2$, and also
$D_1\sqcup D_2=C_{(5)_1}\cup C_{(5)_2}$ and $|D_1|=|D_2|=72$.
Put $T=\underline{D_1}-\underline{D_2}$.
Then a direct verification from the definitions of $D_1,D_2$ gives
\begin{align}
\underline{C_{(2,2)}}\, T &=0, & \underline{C_{(4,2)}}\, T&=0, \label{eq:t_eigs_concrete}\\
\underline{C_{(3)}}\, T &=-5\,T, & \underline{C_{(3,3)}}\, T&=-5\,T, \notag\\
T^2&=144\,\underline{C_0}-18\,\underline{C_{(3)}}-18\,\underline{C_{(3,3)}}+9\,U. \label{eq:t_square_concrete}
\end{align}
The identities in \eqref{eq:t_eigs_concrete} and \eqref{eq:t_square_concrete} have been checked in the GAP file.
Thus all hypotheses \ref{item:lem:main1D}--\ref{item:lem:main4D} of \cref{lem:main} are satisfied.
It follows that $\mathcal{C}^{\tw}_{\mathfrak{A}_6}$ is a Schur partition over $\mathfrak{A}_6$, and that $\mathcal{A}(\mathcal{C}_{\mathfrak{A}_6})$ and $\mathcal{A}(\mathcal{C}^{\tw}_{\mathfrak{A}_6})$ are algebraically isomorphic.

We prove that they are not combinatorially isomorphic by using \cref{lem:inv}.
Let $\mathcal{A}:=\mathcal{A}(\mathcal{C}_{\mathfrak{A}_6})$ and $\mathcal{A}^{\tw}:=\mathcal{A}(\mathcal{C}^{\tw}_{\mathfrak{A}_6})$.
We use the notation $\Gamma_{Y}(X,g)$ for local directed graphs with respect to either $\mathcal{A}$ or $\mathcal{A}^{\tw}$, as determined by the basic sets $X$ and $Y$.
All basic sets appearing in this proof are closed under inverses, so these local directed graphs may be regarded as ordinary undirected graphs.

First consider, in $\mathcal{C}_{\mathfrak{A}_6}$, the local graph $\Gamma_{C_{(3)}}(C_{(5)_1})$.
By $C_{(3)}^{-1}=C_{(3)}$, the graph is undirected.
Moreover the degree of each vertex of $\Gamma_{C_{(3)}}(C_{(5)_1})$ is $p_{(3),(5)_1}^{(5)_1}=5$.
Computing the $C_{(3)}$-neighbors of $x_0=(1\,2\,3\,4\,5)$, we find the five vertices
\[
(1\,2\,4\,5\,3),\
(1\,2\,5\,3\,4),\
(1\,3\,4\,2\,5),\
(1\,4\,2\,3\,5),\
(1\,4\,5\,2\,3).
\]
Since any $y\in C_{(5)_1}$ is $\mathfrak{A}_6$-conjugate to $x_0$, conjugating the same calculation shows that all $C_{(3)}$-neighbors of $y$ have the same fixed point as $y$.
Therefore $\Gamma_{C_{(3)}}(C_{(5)_1})$ decomposes into six subgraphs $\Gamma_1,\Gamma_2,\ldots,\Gamma_6$, according to the fixed point.
A computation further shows that each $\Gamma_k$ is the icosahedral graph.
Thus the connected components of $\Gamma_{C_{(3)}}(C_{(5)_1})$ are six copies of the icosahedral graph.
We denote this graph, the disjoint union of six icosahedral graphs, by $6\operatorname{Icosa}$.
Conjugation by an odd permutation interchanges $C_{(5)_1}$ and $C_{(5)_2}$ while preserving $C_{(3)}$, so $\Gamma_{C_{(3)}}(C_{(5)_2})$ is also isomorphic to $6\operatorname{Icosa}$.

Next consider $\mathcal{C}^{\tw}_{\mathfrak{A}_6}$.
We show that for every $X,Y\in \mathcal{C}^{\tw}_{\mathfrak{A}_6}$, the graph $\Gamma_Y(X)$ is not isomorphic to $6\operatorname{Icosa}$.
Since $6\operatorname{Icosa}$ is a $5$-regular graph on $72$ vertices, the sizes of the parts and the structure constants force $X$ to be either $D_1$ or $D_2$, and $Y$ to be either $C_{(3)}$ or $C_{(3,3)}$.
The GAP computation in \path{A6.g} shows that
$\Gamma_{C_{(3)}}(D_1)$, $\Gamma_{C_{(3,3)}}(D_1)$,
$\Gamma_{C_{(3)}}(D_2)$, and $\Gamma_{C_{(3,3)}}(D_2)$
are connected $5$-regular graphs on $72$ vertices.
Hence, these graphs are not isomorphic to $6\operatorname{Icosa}$.
Therefore, by \cref{lem:inv}, $\mathcal{A}$ and $\mathcal{A}^{\tw}$ are not combinatorially isomorphic.
\end{proof}

We call the association scheme obtained from $\mathcal{C}^{\tw}_{\mathfrak{A}_6}$ the ``twist'' of $\mathfrak{X}(\mathfrak{A}_6)$, and denote it by $\mathfrak{X}^{\tw}(\mathfrak{A}_6)$.

\begin{remark}
  In the construction of \cref{thm:twistedA6}, we used the cut $\{1,2,3\}|\{4,5,6\}$, but the same rule can be used for any cut coming from a $3$-element subset of $\{1,\dots,6\}$.
  Since a subset and its complement give the same cut, this yields the same partition.
  Hence there are $10$ nontrivial partitions of this form, corresponding to the $3+3$ partitions.
  On the other hand, a computer calculation shows that there are no other partitions satisfying \ref{item:lem:main1D}--\ref{item:lem:main4D}.
  It also shows that the association schemes obtained from these ten partitions are all isomorphic.
  Therefore the only non-isomorphic association scheme obtained by this method is $\mathfrak{X}^{\tw}(\mathfrak{A}_6)$.
\end{remark}

\begin{corollary}\label{cor:nonschurian}
$\mathfrak{X}^{\tw}(\mathfrak{A}_6)$ is non-Schurian.
\end{corollary}

\begin{proof}
Consider the local graph $\Gamma_{C_{(3)}}(D_1)$ used in the proof of \cref{thm:twistedA6}.
If $\mathfrak{X}^{\tw}(\mathfrak{A}_6)$ is Schurian, then \cref{lem:schurian-local-transitive} implies that $\Gamma_{C_{(3)}}(D_1)$ is vertex-transitive.

For each vertex $x\in D_1$, consider the number of triangles in $\Gamma_{C_{(3)}}(D_1)$ containing $x$:
\[
\tau(x):=\bigl|\{\{x,y,z\}\subset D_1\mid \{x,y\},\{x,z\},\{y,z\}\in E(\Gamma_{C_{(3)}}(D_1))\}\bigr|.
\]
Computing $\tau(x)$ from the relation matrix by computer, we find that there are $36$ vertices with $\tau(x)=1$ and $36$ vertices with $\tau(x)=3$.
Thus $\Gamma_{C_{(3)}}(D_1)$ is not vertex-transitive.
Therefore $\mathfrak{X}^{\tw}(\mathfrak{A}_6)$ is not Schurian.
\end{proof}

\begin{remark}
Although $\mathfrak A_6\cong\PSL(2,9)$, the construction in this section is not a special case of the $\PSL(2,q)$ construction of \cref{sec:psl_star_switching}.
Indeed, $q=9$ is one of the excluded parameters there, since $m=(9-1)/2=4$ and every unit modulo $m$ is congruent to $\pm1$.
\end{remark}

\subsection{$\mathfrak{X}(\mathfrak{A}_8)$ and a twist}\label{sec:A8}

In this section we consider the group association scheme $\mathfrak{X}(\mathfrak{A}_8)$ obtained from the partition of $\mathfrak{A}_8$ into conjugacy classes.
By repartitioning the union of the two split conjugacy classes of type $(7,1)$ within each cyclic subgroup of order $7$, in a manner analogous to the construction of $D_1,D_2$ for $\mathfrak X(\mathfrak A_6)$, we construct a Cayley association scheme which is algebraically isomorphic to the original association scheme but not combinatorially isomorphic to it.
The group $\mathfrak{A}_8$ has order $20160$, and we use computer calculations in many parts of the following argument.
The computations cited in this section are implemented in \path{A8.g}; the corresponding output is \path{out/result_A8.txt}.

The conjugacy classes of $\mathfrak{A}_8$ are denoted by their cycle types.
By the standard splitting criterion for conjugacy classes in alternating groups, the $\mathfrak S_8$-conjugacy class of cycle type $\lambda$ splits into two conjugacy classes in $\mathfrak{A}_8$ if and only if $\lambda$ consists of distinct odd parts.
Therefore the types $(7,1)$ and $(5,3)$ split in $\mathfrak{A}_8$, while the other types do not.
We denote the split classes by $C_{(7,1)_1}, C_{(7,1)_2}$ and $C_{(5,3)_1}, C_{(5,3)_2}$.
The labels are fixed by requiring
\[
  (1\,2\,3\,4\,5\,6\,7)\in C_{(7,1)_1},
  \qquad
  (1\,2\,3\,4\,5)(6\,7\,8)\in C_{(5,3)_1};
\]
the other split classes of the same types are denoted by $C_{(7,1)_2}$ and $C_{(5,3)_2}$.
With this convention $C_{(7,1)_1}^{-1}=C_{(7,1)_2}$.
In this section we twist only the two split classes $C_{(7,1)_1}$ and $C_{(7,1)_2}$; the split $(5,3)$-classes are left unchanged.

Put
\[
S:=\underline{C_{(7,1)_1}}-\underline{C_{(7,1)_2}},
\qquad
U:=\underline{C_{(7,1)_1}}+\underline{C_{(7,1)_2}}.
\]
A direct computation of products of conjugacy classes shows that, for every conjugacy class $C$ other than $C_{(7,1)_1}$ and $C_{(7,1)_2}$, we have
$\underline{C}\,S=\lambda_C S$.
Explicitly,
\begin{align}
\underline{C_{(3)}}\,S&=0,&
\underline{C_{(3,2,2,1)}}\,S&=0,&
\underline{C_{(3,3,1,1)}}\,S&=0,\notag\\
\underline{C_{(5)}}\,S&=0,&
\underline{C_{(5,3)_1}}\,S&=0,&
\underline{C_{(5,3)_2}}\,S&=0,\notag\\
\underline{C_{(6,2)}}\,S&=0,&
\underline{C_{(2,2)}}\,S&=-14\,S,&
\underline{C_{(2,2,2,2)}}\,S&=-7\,S,\label{eq:A8_S_eigs}\\
\underline{C_{(4,4)}}\,S&=28\,S,&
\underline{C_{(4,2,1,1)}}\,S&=56\,S.\notag
\end{align}
Moreover,
\begin{equation}\label{eq:A8_S_square}
S^2=
-5760\,\underline{C_0}
+384\,\underline{C_{(2,2)}}
+384\,\underline{C_{(2,2,2,2)}}
-128\,\underline{C_{(4,2,1,1)}}
-128\,\underline{C_{(4,4)}}
+64\,U.
\end{equation}
Thus $C_{(7,1)_1}$ and $C_{(7,1)_2}$ satisfy \ref{item:lem:main1}--\ref{item:lem:main4}.

Let $\mathcal{P}$ be the set of all cyclic subgroups of order $7$ in $\mathfrak{A}_8$.
For each $P=\langle x\rangle\in\mathcal{P}$, as in \cref{sec:A6}, write $P^\#$ for the set of nonidentity elements of $P$.
We call $P^\#$ a cell of $P$.
Since the elements of type $(7,1)$ appear as the nonidentity elements of these cells,
\[
\Omega:=C_{(7,1)_1}\sqcup C_{(7,1)_2}
  =\bigsqcup_{P\in\mathcal{P}}P^\#,
  \qquad |\mathcal{P}|=960.
\]
Let $Q=\{1,2,4\}$ and $N=\{3,5,6\}$ be respectively the sets of quadratic residues and quadratic nonresidues in $\F_7^\times=\{1,2,3,4,5,6\}$.
For a chosen generator $x$ of $P\in\mathcal{P}$, set
\[
P^\#_Q(x):=\{x^q\mid q\in Q\},\qquad
P^\#_N(x):=\{x^n\mid n\in N\}.
\]
The unordered pair $\{P^\#_Q(x),P^\#_N(x)\}$ is independent of the generator $x$, while the two labels are interchanged by replacing $x$ by a nonresidue power.
These two sets give a partition of $P^\#$ into two $3$-element sets, and
\[
\{P^\#\cap C_{(7,1)_1},P^\#\cap C_{(7,1)_2}\}=\{P^\#_Q(x), P^\#_N(x)\}.
\]
That is, $C_{(7,1)_1}$ and $C_{(7,1)_2}$ are obtained by consistently choosing, in each cell, one of these two sides.
Since $-Q=N$ in $\mathbb F_7^\times$, taking inverses in $P$ interchanges $P_Q^\#(x)$ and $P_N^\#(x)$.
This implies $C_{(7,1)_1}^{-1}=C_{(7,1)_2}$.

Identify the point set $\{1,\ldots,8\}$ with $V=\F_2^2\times\F_2$ as follows:
\[
\begin{array}{c|cccccccc}
\text{label} & 1&2&3&4&5&6&7&8\\ \hline
(u_1,u_2,z)
&(1,0,0)&(0,1,0)&(1,1,0)&(0,0,1)&(1,0,1)&(0,1,1)&(1,1,1)&(0,0,0).
\end{array}
\]
Consider the affine transformations preserving the pair of parallel affine hyperplanes in $V$
\[
V_0=\{(u_1,u_2,z) \mid z=0\}=\{1,2,3,8\},
\qquad
V_1=\{(u_1,u_2,z) \mid z=1\}=\{4,5,6,7\}.
\]
Explicitly, set
\[
H=
\left\{
(u,z)\longmapsto (Au+b+zc,\ z+\eta)
\ \middle|\
A\in \GL(2,2),\ b,c\in\F_2^2,\ \eta\in\F_2
\right\}.
\]
As a permutation group on the eight letters, $H$ is contained in $G=\mathfrak A_8$, and $|H|=6\cdot4\cdot4\cdot2=192$.
Moreover, its translation subgroup is $C_2^3$, and the stabilizer of the origin is
\[
\left\{
(u,z)\longmapsto (Au+zc,z)
\ \middle|\ A\in\GL(2,2),\ c\in\F_2^2
\right\}
\cong \operatorname{AGL}(2,2)\cong \mathfrak{S}_4.
\]
Thus $H\cong 2^3:\mathfrak{S}_4$.
The containment $H\le G$ follows from the fact that the natural action of $\operatorname{AGL}(3,2)$ on eight points consists of even permutations.
For example, a nontrivial translation is a product of four disjoint transpositions, and the linear part is generated by transvections, each of which is an even permutation.

The group $H$ acts on $\mathcal P$ by conjugation.
This action has seven orbits, represented by the cells generated by the following $7$-cycles:
\[
\begin{array}{lll}
r_1=(1\,2\,7\,3\,5\,6\,4),&
 r_2=(1\,2\,7\,6\,3\,5\,4),&
 r_3=(1\,2\,7\,5\,6\,3\,4),\\
r_4=(1\,2\,3\,7\,6\,5\,4),&
 r_5=(1\,2\,3\,7\,5\,6\,4),&
 r_6=(1\,2\,6\,3\,5\,7\,4),\\
r_7=(1\,2\,5\,3\,7\,6\,4).&&
\end{array}
\]
The sizes of the $H$-orbits of $\langle r_i\rangle$ are, in order,
$192,192,64,192,192,64,64$, whose sum is $960=|\mathcal P|$.

We orient these seven representative cells by the $Q$-side and define
\[
D_1
=
\bigcup_{i=1}^{7}\ \bigcup_{h\in H}
\{\,h r_i^q h^{-1}\mid q\in Q\,\}, \quad
D_2
=
\bigcup_{i=1}^{7}\ \bigcup_{h\in H}
\{\,h r_i^n h^{-1}\mid n\in N\,\}.
\]
This definition does not depend on how the same cell is represented.
Indeed, suppose that a cell $P$ belonging to the same $H$-orbit as $\langle r_i\rangle$ is represented in two ways:
\[
P=h\langle r_i\rangle h^{-1}=h'\langle r_i\rangle h'^{-1}.
\]
Then $a:=h^{-1}h'$ normalizes $\langle r_i\rangle$.
Thus there exists $\kappa\in\F_7^\times$ such that $a r_i a^{-1}=r_i^\kappa$.
Since $a\in G$, the elements $r_i$ and $a r_i a^{-1}$ belong to the same $G$-conjugacy class; hence $\kappa\in Q$.
Therefore, if $q\in Q$ then $\kappa q\in Q$, and if $n\in N$ then $\kappa n\in N$.
Thus
\[
h' r_i^q h'^{-1}=h r_i^{\kappa q}h^{-1}\quad(q\in Q),
\qquad
h' r_i^n h'^{-1}=h r_i^{\kappa n}h^{-1}\quad(n\in N).
\]
Consequently, the three elements of a fixed cell assigned to the $Q$-side and the three elements assigned to the $N$-side do not depend on the choice of representative.
Hence each cell $P\in\mathcal P$ contributes three elements to $D_1$ and three elements to $D_2$.
Therefore
\[
D_1\sqcup D_2=\Omega,
\qquad
|D_1|=|D_2|=2880.
\]
Moreover, since $N=-Q$ in $\F_7^\times$, we have $D_1^{-1}=D_2$.

\begin{remark}\label{rem:A8_binary_relabel}
The orientations of the seven representative cells above are part of the construction; they specify, in each cell, which of the $Q$- and $N$-sides is assigned to $D_1$ and which is assigned to $D_2$.
Even for a $7$-cycle representing the same $H$-orbit, replacing the representative by a nonresidue power interchanges $D_1$ and $D_2$ throughout that orbit.
Thus the representatives cannot be chosen arbitrarily.
\end{remark}

For the sets $D_1,D_2$ defined above, put
\[
\mathcal{C}^{\tw}_{\mathfrak{A}_8}
:=\{C\mid C\text{ is a conjugacy class of }\mathfrak{A}_8
\text{ other than }C_{(7,1)_1}\text{ and }C_{(7,1)_2}\}
  \cup\{D_1,D_2\}.
\]
\begin{theorem}\label{thm:A8_twisted_exists}
$\mathcal{C}^{\tw}_{\mathfrak{A}_8}$ is a Schur partition over $\mathfrak{A}_8$.
Moreover, $\mathcal{C}_{\mathfrak{A}_8}$ and $\mathcal{C}^{\tw}_{\mathfrak{A}_8}$ are algebraically isomorphic, but not combinatorially isomorphic.
\end{theorem}

\begin{proof}
Put $T=\underline{D_1}-\underline{D_2}$.
By construction, $D_1\sqcup D_2=C_{(7,1)_1}\cup C_{(7,1)_2}$, $|D_1|=|D_2|$, and $D_1^{-1}=D_2$.
A direct computer verification, in the same way as in \cref{thm:twistedA6}, gives
\[
  \underline{C}\,T=\lambda_C T
\]
for every conjugacy class $C$ other than $C_{(7,1)_1}$ and $C_{(7,1)_2}$, with the same eigenvalues $\lambda_C$ as in \eqref{eq:A8_S_eigs}, and also gives $T^2=S^2$, where $S^2$ is given in \eqref{eq:A8_S_square}.
Thus $\mathcal{C}^{\tw}_{\mathfrak{A}_8}$ satisfies the hypotheses of \Cref{lem:main}.
Therefore $\mathcal{C}^{\tw}_{\mathfrak{A}_8}$ is a Schur partition over $\mathfrak{A}_8$, and $\mathcal{A}(\mathcal{C}_{\mathfrak{A}_8})$ and $\mathcal{A}(\mathcal{C}^{\tw}_{\mathfrak{A}_8})$ are algebraically isomorphic.

To show that they are not combinatorially isomorphic, we use the fact that the number of connected components of $\Gamma_{C_{(3)}}(C_{(7,1)_r})$ is an isomorphism invariant.
For the original partition $\mathcal{C}_{\mathfrak{A}_8}$, a computer calculation confirms that $\Gamma_{C_{(3)}}(C_{(7,1)_r})$ is a $21$-regular graph on $2880$ vertices, and is the disjoint union of eight connected components according to the fixed point of the $7$-cycle.
On the other hand, in $\mathcal{C}^{\tw}_{\mathfrak{A}_8}$, the only local graphs which are $21$-regular graphs on $2880$ vertices are $\Gamma_{C_{(3)}}(D_1)$ and $\Gamma_{C_{(3)}}(D_2)$, and a computer calculation confirms that both are connected.
Thus, by \cref{lem:inv}, $\mathcal{A}(\mathcal{C}_{\mathfrak{A}_8})$ and $\mathcal{A}(\mathcal{C}^{\tw}_{\mathfrak{A}_8})$ are not combinatorially isomorphic.
\end{proof}

We call the association scheme obtained from $\mathcal{C}^{\tw}_{\mathfrak{A}_8}$ the ``twist'' of $\mathfrak{X}(\mathfrak{A}_8)$, and denote it by $\mathfrak{X}^{\tw}(\mathfrak{A}_8)$.

\begin{corollary}\label{cor:A8_nonSchurian}
$\mathfrak{X}^{\tw}(\mathfrak{A}_8)$ is non-Schurian.
\end{corollary}

\begin{proof}
Consider the local graph $\Gamma_{C_{(3)}}(D_1)$.
A computer calculation confirms that the number of triangles containing a vertex is not constant on $D_1$; for instance, the values $27$ and $54$ both occur.
Thus $\Gamma_{C_{(3)}}(D_1)$ is not vertex-transitive.
If $\mathfrak{X}^{\tw}(\mathfrak{A}_8)$ were Schurian, then \cref{lem:schurian-local-transitive} would imply that this local graph is vertex-transitive, a contradiction.
\end{proof}

\subsection{The cases $n=4,5,7,9$}\label{sec:An}

In this subsection we use \cref{cor:unique} to verify, for $n=4,5,7,9$, that no nontrivial twist obtained by repartitioning split conjugacy classes exists for $\mathfrak{A}_n$.
The finite computations used in this section are implemented in \path{A4.g}, \path{A5.g}, \path{A7.g}, and \path{A9.g}; the corresponding outputs are \path{out/result_A4.txt}, \path{out/result_A5.txt}, \path{out/result_A7.txt}, and \path{out/result_A9.txt}.

\subsubsection{The case of $\mathfrak{A}_4$}
In $\mathfrak A_4$, the two classes of $3$-cycles split as $C_{(3)_1}$ and $C_{(3)_2}$, each of size $4$.
These classes satisfy \ref{item:lem:main1}--\ref{item:lem:main4}.
Put
\[
\Omega:=C_{(3)_1}\sqcup C_{(3)_2}.
\]
We use the criterion of \cref{cor:unique} for the graph $\Gamma_\Omega$ on this $\Omega$.
The graph $\Gamma_\Omega$ is the complete bipartite graph $K_{4,4}$.
Indeed, by writing down the eight $3$-cycles of $\mathfrak{A}_4$ explicitly, one sees that for any $x,y\in C_{(3)_1}$, the element $yx^{-1}$ is not a $3$-cycle, whereas for any $x\in C_{(3)_1}$ and $y\in C_{(3)_2}$, one has $yx^{-1}\in \Omega$.
Thus $\Gamma_\Omega$ is the complete bipartite graph with parts $C_{(3)_1}$ and $C_{(3)_2}$.
Let $A_\Omega$ be its adjacency matrix.
Consequently the spectrum of $A_\Omega$ is $\{4^1,0^6,-4^1\}$.

\begin{proposition}\label{thm:A4_rigidity}
Let $D_1,D_2$ be a partition of $\Omega$ satisfying assumptions \ref{item:lem:main1D}--\ref{item:lem:main4D} of \cref{lem:main}.
Then
\[
\{D_1,D_2\}=\{C_{(3)_1},C_{(3)_2}\}.
\]
In particular, there is no nontrivial twist obtained by repartitioning the split $3$-cycle classes.
\end{proposition}

\begin{proof}
By \cref{lem:main2}, the eigenvalue by which $A_\Omega$ acts on the original difference vector $\chi_S$ is $\lambda_U$, and a calculation gives $\lambda_U=-4$.
On the other hand, the spectrum of $A_\Omega$ shows that the $(-4)$-eigenspace is one-dimensional.
The conclusion follows from \cref{cor:unique}.
\end{proof}

\subsubsection{The case of $\mathfrak{A}_5$}
\label{sec:A5}

In $\mathfrak A_5$, the only split conjugacy classes are the two $5$-cycle classes $C_{(5)_1}$ and $C_{(5)_2}$, each of size $12$.
They satisfy \ref{item:lem:main1}--\ref{item:lem:main4}; this is verified from the class multiplication coefficients.
Put
\[
\Omega:=C_{(5)_1}\sqcup C_{(5)_2}.
\]

Consider the graph $\Gamma_{(2,2)}$ on $\Omega$ associated with the conjugacy class $C_{(2,2)}$, and write its adjacency matrix as $A_{(2,2)}$.
A computer calculation shows that the difference vector $\chi_S$ is an eigenvector of $A_{(2,2)}$ with eigenvalue $-5$, and that the characteristic polynomial of $A_{(2,2)}$ is
\[
\chi_{A_{(2,2)}}(\lambda)
=(\lambda-5)(\lambda+5)(\lambda^2-5)^6(\lambda^2-1)^5.
\]
Thus, as in \cref{thm:A4_rigidity}, we obtain the following.

\begin{proposition}\label{thm:A5_rigidity}
Let $D_1,D_2$ be a partition of $\Omega$ satisfying assumptions \ref{item:lem:main1D}--\ref{item:lem:main4D} of \cref{lem:main}.
Then
\[
\{D_1,D_2\}=\{C_{(5)_1},C_{(5)_2}\}.
\]
In particular, there is no nontrivial twist obtained by repartitioning the split $5$-cycle classes.
\end{proposition}

\begin{proof}
By \cref{lem:main2}, the eigenvalue by which $A_{(2,2)}$ acts on the original difference vector $\chi_S$ is $\lambda_{(2,2)}$, and a calculation gives $\lambda_{(2,2)}=-5$.
On the other hand, the displayed characteristic polynomial shows that the $(-5)$-eigenspace of $A_{(2,2)}$ is one-dimensional.
The conclusion follows from \cref{cor:unique}.
\end{proof}

\begin{remark}
It is already known that the group association scheme $\mathfrak{X}(\mathfrak{A}_4)$ is uniquely determined up to combinatorial isomorphism by its intersection numbers, by the classification of association schemes on $12$ points \cite{Hirasaka12} and by the Hanaki--Miyamoto database of association schemes \cite{HanakiMiyamotoSmallOrderWeb}.
It is also already known, by Tomiyama \cite{TomiyamaA5}, that the group association scheme $\mathfrak{X}(\mathfrak{A}_5)$ is uniquely determined up to combinatorial isomorphism by its intersection numbers.
Thus, in both cases, the intersection numbers determine the corresponding group association scheme up to combinatorial isomorphism.
The results in this section are therefore weaker than what is already known, but they rederive the corresponding rigidity using \cref{cor:unique}.
\end{remark}

\subsubsection{The case of $\mathfrak{A}_7$}\label{sec:A7}
In $\mathfrak A_7$, the only split conjugacy classes are the two $7$-cycle classes $C_{(7)_1}$ and $C_{(7)_2}$, each of size $360$.
These classes $C_{(7)_1},C_{(7)_2}$ satisfy \ref{item:lem:main1}--\ref{item:lem:main4}.
Put
\[
\Omega:=C_{(7)_1}\sqcup C_{(7)_2}.
\]

Consider the graph $\Gamma_{(2,2)}$ on $\Omega$ associated with the conjugacy class $C_{(2,2)}$, and write its adjacency matrix as $A_{(2,2)}$.
A computer calculation shows that the difference vector $\chi_S$ is an eigenvector of $A_{(2,2)}$ with eigenvalue $-21$, and that the eigenspace of $A_{(2,2)}$ for the eigenvalue $-21$ is one-dimensional.
Thus, as in \cref{thm:A4_rigidity}, we obtain the following.

\begin{proposition}\label{thm:A7_rigidity}
Let $D_1,D_2$ be a partition of $\Omega$ satisfying assumptions \ref{item:lem:main1D}--\ref{item:lem:main4D} of \cref{lem:main}.
Then
\[
\{D_1,D_2\}=\{C_{(7)_1},C_{(7)_2}\}.
\]
In particular, there is no nontrivial twist obtained by repartitioning the split $7$-cycle classes.
\end{proposition}

\begin{proof}
By \cref{lem:main2}, the eigenvalue by which $A_{(2,2)}$ acts on the original difference vector $\chi_S$ is $\lambda_{(2,2)}$, and a calculation gives $\lambda_{(2,2)}=-21$.
On the other hand, the eigenspace of $A_{(2,2)}$ for the eigenvalue $-21$ is one-dimensional.
The conclusion follows from \cref{cor:unique}.
\end{proof}

\subsubsection{The case of $\mathfrak{A}_9$}\label{sec:A9}

The split conjugacy types in $\mathfrak A_9$ are the $9$-cycle classes and the classes of type $(5,3)$.
The computation in \path{A9.g} checks that these are the only pairs needed for \cref{cor:unique}; the other equal-size pair $C_{(3,2,2)},C_{(4,2)}$ does not satisfy \ref{item:lem:main1}--\ref{item:lem:main4}.

First consider the $9$-cycle case.
The classes $C_{(9)_1},C_{(9)_2}$ satisfy \ref{item:lem:main1}--\ref{item:lem:main4}.
Put
\[
\Omega:=C_{(9)_1}\sqcup C_{(9)_2}.
\]
Consider the graph $\Gamma_{(2,2)}$ on $\Omega$ associated with the conjugacy class $C_{(2,2)}$, and write its adjacency matrix as $A_{(2,2)}$.
A computer calculation shows that the difference vector $\chi_S$ is an eigenvector of $A_{(2,2)}$ with eigenvalue $-54$, and that the eigenspace of $A_{(2,2)}$ for the eigenvalue $-54$ is one-dimensional.

\begin{proposition}\label{thm:A9_9_rigidity}
Let $D_1,D_2$ be a partition of $\Omega$ satisfying assumptions \ref{item:lem:main1D}--\ref{item:lem:main4D} of \cref{lem:main}.
Then
\[
\{D_1,D_2\}=\{C_{(9)_1},C_{(9)_2}\}.
\]
In particular, there is no nontrivial twist obtained by repartitioning the split $9$-cycle classes.
\end{proposition}

\begin{proof}
By \cref{lem:main2}, the eigenvalue by which $A_{(2,2)}$ acts on the original difference vector $\chi_S$ is $\lambda_{(2,2)}$, and a calculation gives $\lambda_{(2,2)}=-54$.
On the other hand, the eigenspace of $A_{(2,2)}$ for the eigenvalue $-54$ is one-dimensional.
The conclusion follows from \cref{cor:unique}.
\end{proof}

Next consider the type $(5,3)$ case.
The classes $C_{(5,3)_1},C_{(5,3)_2}$ satisfy \ref{item:lem:main1}--\ref{item:lem:main4}.
Put
\[
\Omega:=C_{(5,3)_1}\sqcup C_{(5,3)_2}.
\]
Consider the graph $\Gamma_{(3)}$ on $\Omega$ associated with the conjugacy class $C_{(3)}$, and write its adjacency matrix as $A_{(3)}$.
A computer calculation shows that the difference vector $\chi_S$ is an eigenvector of $A_{(3)}$ with eigenvalue $-24$, and that the eigenspace of $A_{(3)}$ for the eigenvalue $-24$ is one-dimensional.

\begin{proposition}\label{thm:A9_53_rigidity}
Let $D_1,D_2$ be a partition of $\Omega$ satisfying assumptions \ref{item:lem:main1D}--\ref{item:lem:main4D} of \cref{lem:main}.
Then
\[
\{D_1,D_2\}=\{C_{(5,3)_1},C_{(5,3)_2}\}.
\]
In particular, there is no nontrivial twist obtained by repartitioning the split $(5,3)$ classes.
\end{proposition}

\begin{proof}
By \cref{lem:main2}, the eigenvalue by which $A_{(3)}$ acts on the original difference vector $\chi_S$ is $\lambda_{(3)}$, and a calculation gives $\lambda_{(3)}=-24$.
On the other hand, the eigenspace of $A_{(3)}$ for the eigenvalue $-24$ is one-dimensional.
The conclusion follows from \cref{cor:unique}.
\end{proof}

\begin{remark}
The results in \Cref{sec:A7,sec:A9} show that, at least in $\mathfrak{A}_7$ and $\mathfrak{A}_9$, \cref{lem:main} does not work.
In this sense, in contrast to $\mathfrak{A}_6$ and $\mathfrak{A}_8$, the groups $\mathfrak{A}_7$ and $\mathfrak{A}_9$ are rigid with respect to \Cref{lem:main}.
\end{remark}

\section{Checking the existence of twists for groups of order at most $200$}\label{sec:other_finite_groups_twist}

In this section, we apply the construction in \cref{sec:twisted_conjugacy} to all isomorphism classes of groups of order at most $200$.
Using the Small Groups Library in GAP, we first search for pairs of conjugacy classes satisfying \ref{item:lem:main1}--\ref{item:lem:main4}.
For each such pair, we perform an exhaustive search for nontrivial partitions $D_1,D_2$ satisfying assumptions \ref{item:lem:main1D}--\ref{item:lem:main4D} of \cref{lem:main}, where nontrivial means $\{D_1,D_2\}\ne \{C_{i_1},C_{i_2}\}$.

The computation used \path{small_groups_twist.g}.
Detailed results for every class pair satisfying \ref{item:lem:main1}--\ref{item:lem:main4} are recorded in \path{out/result_small_groups_twist.csv}, and the summary is recorded in \path{out/result_small_groups_twist.txt}.
The search examines all $6{,}065$ group isomorphism classes in this range; $1{,}672$ of them have at least one pair satisfying \ref{item:lem:main1}--\ref{item:lem:main4}.
It finds $556$ nontrivial unordered partitions $\{D_1,D_2\}$ satisfying the hypotheses of \cref{lem:main}.
Of these, $256$ yield Schur partitions that are not combinatorially isomorphic to the Schur partition formed by the conjugacy classes of the same group.
These new twists occur for $48$ group isomorphism types, of orders $108$, $128$, $144$, $160$, $192$, and $200$; the smallest example is $\mathrm{SmallGroup}(108,15)$.
The number $256$ counts partitions, rather than combinatorial isomorphism classes of the association schemes arising from these partitions.

For $G=\mathrm{SmallGroup}(24,12)\simeq\mathfrak{S}_4$, the computation finds three non-original partitions, but all three associated Cayley schemes are combinatorially isomorphic to $\mathfrak X(\mathfrak{S}_4)$.
This is compatible with Tomiyama's classification \cite{TomiyamaA5}: up to isomorphism, exactly three association schemes have the same intersection numbers as $\mathfrak X(\mathfrak{S}_4)$.
The other two isomorphism classes have no regular subgroup in their automorphism groups and hence are not Cayley association schemes.

\section{Summary and future questions}\label{sec:summary}
The main result of this paper is the construction for $\PSL(2,q)$ obtained by twisting the sets $H_s$ inside a Borel subgroup.
We constructed a Schur partition $\calD_{q,k}$ of $\PSL(2,q)$ which is algebraically isomorphic to the partition of $\PSL(2,q)$ into conjugacy classes but whose Cayley association scheme is not combinatorially isomorphic to the group association scheme.
This gives an infinite family of primitive group association schemes arising from finite simple groups that are not determined up to combinatorial isomorphism by their intersection numbers.
In particular, these group association schemes are non-separable.

For $\mathfrak A_6$ and $\mathfrak A_8$, the construction in \cref{sec:twisted_conjugacy} gives Schur partitions that are algebraically isomorphic but not combinatorially isomorphic to the corresponding partitions into conjugacy classes.
Thus the corresponding group association schemes are not determined up to combinatorial isomorphism by their intersection numbers and are non-separable.
The resulting twisted association schemes are also non-Schurian.
For $\mathfrak A_4,\mathfrak A_5,\mathfrak A_7$, and $\mathfrak A_9$, \cref{cor:unique} shows that only the original partitions satisfy the hypotheses of \cref{lem:main}.
The search over groups of order at most $200$ finds $256$ nontrivial unordered partitions whose associated schemes are not combinatorially isomorphic to the original group association schemes, occurring for $48$ group isomorphism types.

\begin{question}
For which non-abelian finite simple groups $G$ is the group association
scheme $\mathfrak X(G)$ not determined up to combinatorial isomorphism
by its intersection numbers?
\end{question}

Up to the isomorphism $\PSL(2,4)\cong\PSL(2,5)$, the excluded values of $q$ for which $\PSL(2,q)$ is non-abelian simple are $q=5,7,9,13$.
For these values, the construction in \cref{sec:psl_star_switching} gives no new partition.
The cases $q=5,7,9$ are already settled.
For $q=5$ and $q=7$, Tomiyama proved that the corresponding group association schemes are uniquely determined up to combinatorial isomorphism by their intersection numbers \cite{TomiyamaA5,TomiyamaPSL27}.
For $q=9$, we have $\PSL(2,9)\cong\mathfrak A_6$, and \cref{thm:twistedA6} gives a nontrivial twist.
Thus only $q=13$ remains open.
\begin{question}
Is the group association scheme $\mathfrak X(\PSL(2,13))$ uniquely determined up to combinatorial isomorphism by its intersection numbers?
In particular, can a different switching construction yield a counterexample?
\end{question}

\begin{question}
For $\PSL(2,q)$, $\mathfrak A_6$ and $\mathfrak A_8$, are there further twists, not equivalent to the ones constructed here, which are algebraically isomorphic but combinatorially non-isomorphic to the original group association schemes?
Do there exist association schemes algebraically isomorphic to $\mathfrak{X}(\PSL(2,q))$, $\mathfrak X(\mathfrak A_6)$ or $\mathfrak X(\mathfrak A_8)$ which are not Cayley association schemes?
\end{question}


\end{document}